\documentclass[a4paper,12pt]{amsart}

\usepackage{amssymb,amsbsy,amsmath,amsfonts,amssymb,amscd}
\usepackage{latexsym}
\usepackage{graphics}
\usepackage{color}
\input xy
\xyoption{all}

\newcommand\sC{{\mathcal C}}
\newcommand\sT{{\mathcal T}}
\newcommand\sD{{\mathcal D}}
\newcommand\sE{{\mathcal E}}
\newcommand\sA{{\mathcal A}}

\newcommand\sG{{\mathcal G}}
\newcommand\sI{{\mathcal I}}

\newcommand\sL{{\mathcal L}}
\newcommand\sZ{{\mathcal Z}}
\newcommand\sB{{\mathcal B}}
\newcommand\sN{{\mathcal N}}

\newcommand\sX{{\mathcal X}}
\newcommand\sY{{\mathcal Y}}
\newcommand\sH{{\mathcal H}}

                 \newcommand\sM{{\mathcal M}}
                 \newcommand\XX{{\mathfrak X}}

\newcommand\Ga{\Gamma}
\newcommand\De{\Delta}

\newcommand\de{\delta}

\DeclareMathOperator{\Def}{Def}

\def\Bbb{\bf}
\newcommand{\CC}{\ensuremath{\mathbb{C}}}
\newcommand{\RR}{\ensuremath{\mathbb{R}}}
\newcommand{\ZZ}{\ensuremath{\mathbb{Z}}}
\newcommand{\QQ}{\ensuremath{\mathbb{Q}}}

\newcommand{\sS}{\ensuremath{\mathcal{S}}}

\newcommand{\hol}{\ensuremath{\mathcal{O}}}

\newcommand{\PP}{\ensuremath{\mathbb{P}}}

\newcommand{\FF}{\ensuremath{\mathbb{F}}}
\newcommand{\HHH}{\ensuremath{\mathcal{H}}}

\newcommand{\ra}{\ensuremath{\rightarrow}}

\def\eea{\end{eqnarray*}}
\def\bea{\begin{eqnarray*}}

\newcommand\dual{\mathrel{\raise3pt\hbox{$\underline{\mathrm{\thinspace d
\thinspace}}$}}}
\newcommand\qe{\ifhmode\unskip\nobreak\fi\quad $\Box$}       

\def\BOX{\hfill\lower.5\baselineskip\hbox{$\Box$}}

\newtheorem{theo}[equation]{Theorem}
\newtheorem{remarkk}[equation]{Remark}
\newenvironment{rem}{\begin{remarkk}\rm}{\end{remarkk}}

\newtheorem{defin}[equation]{Definition}

\newtheorem{prop}[equation]{Proposition}
\newtheorem{cor}[equation]{Corollary}
\newtheorem{lemma}[equation]{Lemma}
\newtheorem{example}[equation]{Example}

\newcommand{\SSS}{\ensuremath{\mathcal{S}}}
\newcommand{\sR}{\ensuremath{\mathcal{R}}}
\newcommand{\X}{\ensuremath{\mathcal{X}}}
\newcommand{\Y}{\ensuremath{\mathcal{Y}}}
\newcommand{\T}{\ensuremath{\mathbb{T}}}

\def\C{{\Bbb C}}

\newcommand{\Proof}{{\it Proof. }}
\begin{document}

\title[Deformation and moduli]{ A superficial  working guide to deformations and moduli}
\author{ F. Catanese}\footnote{I owe to David Buchsbaum the joke that
 an expert on algebraic surfaces is  a `superficial' mathematician. }
\address {Lehrstuhl Mathematik VIII\\
Mathematisches Institut der Universit\"at Bayreuth\\
NW II,  Universit\"atsstr. 30\\
95447 Bayreuth}
\email{fabrizio.catanese@uni-bayreuth.de}

\thanks{The present work took place in the realm of the DFG
Forschergruppe 790 ``Classification of algebraic
surfaces and compact complex manifolds''.
A major part of the article was written, and the article was completed, when the author was a visiting research scholar at KIAS. }


\maketitle

{\em Dedicated to David Mumford  with admiration.}

\tableofcontents
\section*{Introduction}

There are several ways to look at moduli theory, indeed the same name 
can at a first glance disguise
completely different approaches to mathematical thinking; yet there is a substantial unity since,
although often with different languages and purposes,
the problems treated are substantially the same.

The most classical approach and motivation  is to consider moduli theory as the fine part 
of classification theory:
the big quest is not just to prove that certain moduli spaces  exist, but to 
use the study of  their structure in order to obtain geometrical informations
about the varieties one wants to classify; and using
each time  the most convenient incarnation of `moduli'.

For instance, as a slogan, we might think of   moduli theory and deformation theory
as analogues of the global study of an algebraic variety versus a local study of
its singularities, done using power series methods. On the other hand, the shape of an algebraic variety
is easily recognized when it has singularities!

In this article most of our attention will be cast on the case of 
complex algebraic surfaces, which is
already sufficiently intricate  to defy many attempts of investigation.
But we shall try, as much as possible, to treat the higher dimensional and more general cases as well.
We shall also stick to the world of complex manifolds and complex 
projective varieties,
which allows us to find so many beautiful connections to related 
fields of mathematics,
such as topology, differential geometry and symplectic geometry.

David Mumford clarified the concept of biregular moduli  through  a 
functorial definition,
which is extremely useful when we want a precise answer to  questions 
concerning
a certain class of algebraic varieties.  

The underlying elementary concepts  are the concept of normal forms, and of quotients of parameter spaces
by a suitable equivalence relation, often given by the action of
an appropriate group.
To give an idea through an elementary geometric problem: 
how many are the projective equivalence classes of smooth plane curves of degree
4 admitting 4 distinct collinear hyperflexes?

A birational approach to moduli existed before, since, by the work of 
Cayley, Bertini, Chow and van der Waerden,
varieties $X^n_d \subset \PP^N$ in a fixed projective space,  having 
a fixed  dimension $n$ and  a fixed degree
$d$ are parametrized by the so called Chow variety $\sC h(n;d;N)$, over 
which the projective group $G:= \PP GL (N+1,
\CC)$ acts. And, if $Z$ is an irreducible component of $\sC h(n;d;N)$, 
the transcendence degree of the field of
invariant rational functions $ \CC(Z)^G$ was classically called the 
number of polarized moduli for the
class of varieties parametrized by $Z$. This topic: `embedded varieties' is treated in the article by Joe Harris
in this Handbook.

A typical example leading to the concept of stability was: take the fourfold symmetric product $Z$ of 
$\PP^2$, parametrizing
4-tuples of points in the plane. Then $Z$ has dimension 8 and the 
field of invariants has transcendence degree 0.
This is not a surprise, since 4 points in linear general position are 
a projective basis, hence they are projectively
equivalent; but, if one takes 4 point to lie on a line, then there is 
a modulus, namely, the cross ratio.
This example, plus the other basic example given by the theory of Jordan normal forms 
of square matrices
(explained in \cite{Mum-Suom} in detail) guide our understanding of 
the basic problem of Geometric Invariant
Theory: in which sense may we consider the quotient of a 
variety by the action of an algebraic group.
In my opinion geometric invariant theory, in spite of its beauty and its conceptual simplicity, but in view of its difficulty,
is a foundational but  not a fundamental  tool 
in classification theory.
Indeed one of the most difficult results, due to Gieseker, is the 
asymptotic stability of
pluricanonical images of surfaces of general type; it has as an
important corollary the existence of a moduli space for the
canonical models of surfaces of general type, but the methods of 
proof do not shed light on
the classification of such surfaces (indeed boundedness for the 
families of surfaces with given invariants
had followed earlier by the results of Moishezon, Kodaira and Bombieri).

We use in our title the name `working': this may mean many things, but 
in particular here our goal is to show how
to use the methods of deformation theory in order to classify 
surfaces with given invariants.

The order in our exposition is more guided by historical development 
and by our education than by a stringent
logical nesting.

The first guiding concepts are the concepts of Teichm\"uller space and moduli space associated to an oriented
compact differentiable manifold $M$ of even dimension. These however are only defined as
topological spaces, and one needs the Kodaira-Spencer-Kuranishi theory in order to try to give the structure of
a complex space to them.

A first question which we investigate, and about which we give some new results (proposition \ref{kur=teich} and theorem
\ref{kur=teich-surf}), is:  when is Teichm\"uller space
locally homeomorphic to Kuranishi space?

 This equality has been often taken for granted, of course  under the assumption of the validity of
the so called Wavrik condition (see theorem \ref{kur3}), which  requires the dimension of the space of holomorphic vector fields to
be locally constant under deformation .

An important role plays the example of Atiyah about surfaces acquiring a node: we interpret it here as showing that Teichm\"uller space
is non separated (theorem \ref{nonseparated}). In section 4 we see that it also underlies some recent pathological behaviour of automorphisms of surfaces,
recently discovered together with Ingrid Bauer: even if deformations of canonical and minimal models are essentially the same, up to 
finite base change, the same does not occur for deformations of automorphisms (theorems \ref{main1} and \ref{path}). The connected components for deformation of
automorphisms of canonical models $(X,G,\alpha)$ are bigger than the connected components for deformation of
automorphisms of minimal models $(S,G,\alpha')$, the latter yielding locally closed sets of the moduli spaces which
are locally closed but not closed.

To describe these results we explain first the Gieseker coarse moduli space for canonical models of surfaces of general type,
which has the same underlying reduced space as  the coarse moduli stack for minimal models of surfaces of general type. We do not essentially talk about stacks (for which an elementary presentation can be found in \cite{fantechi}),
but we clarify how moduli spaces are obtained by glueing together Kuranishi spaces, and we show the fundamental difference for the \'etale
equivalence relation in the two respective cases of canonical and minimal models: we exhibit examples  showing that the relation is not finite (proper) in the case of minimal models (a fact which underlies the definition of Artin stacks given in \cite{artinstacks}).

We cannot completely settle here the question whether Teichm\"uller space is 
locally homeomorphic to Kuranishi space for all surfaces of general type, as this question is related to a fundamental question
about the non existence of complex automorphisms which are isotopic to the identity, but different from  the identity (see however 
the already mentioned theorem \ref{kur=teich-surf}).

Chapter five is dedicated to the connected components of moduli spaces, and to  the action of the absolute Galois group on the
set of irreducible components of the moduli space, and surveys many recent
results.

We end by discussing concrete issues showing how  one can determine a connected component of the moduli space by  resorting to topological or differential arguments; we overview  several results, without proofs but citing the references, and finally  we prove a new result, theorem \ref{doublecover}, obtained in collaboration with Ingrid Bauer.

There would have been many other interesting topics to treat, but these should probably better belong to a  `part 2' of the working  guide.

\section{Analytic moduli spaces and local moduli spaces: 
Teichm\"uller and Kuranishi space}

\subsection{Teichm\"uller space}
 Consider, throughout this subsection, an  oriented real differentiable manifold $M$ of real dimension $2n$
 (without loss of generality we may a posteriori assume $M$ and all the rest 
 to be $\sC^{\infty}$ or even $\sC^{\omega}$, i.e.,  real-analytic).
 
 At a later point it will be convenient to assume that $M$ is compact.
 
 Ehresmann (\cite{ACS}) defined an {\bf almost complex structure} on $M$ as the structure
 of a complex vector bundle on the real tangent bundle $TM_{\RR}$: namely, the action of 
 $\sqrt {-1}$ on $TM_{\RR}$ is provided by an endomorphism
 $$ J : TM_{\RR} \ra TM_{\RR}, {\rm{\  with  }}\  J^2 = - Id.$$

It is completely equivalent to give the decomposition of the complexified tangent bundle $TM_{\CC} : =  TM_{\RR} \otimes_{\RR}\CC$
as the direct sum of the $i$, respectively $-i$ eigenbundles:
$$  TM_{\CC}  =  TM^{1,0} \oplus  TM^{0,1} {\rm{ \  where  }} \  TM^{0,1} = \overline {TM^{1,0} }.$$

In view of the second condition, it suffices to give the subbundle $TM^{1,0} $, or, equivalently, a section
of the associated Grassmannian bundle $ \sG (n, TM_{\CC})$ whose fibre at a point $x \in M$
is the variety of $n$-dimensional vector subspaces of the complex tangent space at $x$, $TM_{\CC, x}$
(note that  the section must take values in the open set  $\mathcal T_n$ of subspaces $V$ such that $V$ and $\bar{V}$ generate).

The space $\sA \sC (M)$ of almost complex structures, once  $TM_{\RR}$ (hence all associated bundles) is endowed with a
Riemannian metric, has a countable number of seminorms (locally, the sup norm on a compact $K$
of all the derivatives of the  endomorphism $J$), and is therefore  a Fr\'echet space. One may for instance assume that $M$ is embedded in some $\RR^N$.

Assuming   that $M$ is compact, one can also consider the Sobolev k-norms (i.e., for derivatives up order k).

A closed subspace of $\sA \sC (M)$ consists of the set $ \sC (M)$ of complex structures: these are the almost complex structures for
which there are at each point $x$ local holomorphic coordinates, i.e., functions $z_1, \dots , z_n$ whose differentials
span the dual $(TM^{1,0}_y)^{\vee}$ of $TM^{1,0}_y$ for each point $y$ in a neighbourhood of $x$.

In general, the  splitting 
$$TM_{\CC}^{\vee} = (TM^{1,0})^{\vee} \oplus (TM^{0,1})^{\vee}$$ yields a decomposition of exterior differentiation
 of functions as
$ df = \partial f + \bar{\partial} f$, and a function is said to be holomorphic if its
differential is complex linear, i.e., $ \bar{\partial} f = 0$. 

This decomposition $ d= \partial  + \bar{\partial} $
extends to higher degree differential forms.

The theorem of Newlander-Nirenberg (\cite{NN}), first proven by Eckmann and Fr\"olicher in the real analytic case (\cite{E-F}, see 
also \cite{montecatini} for a simple proof)
characterizes the complex structures through an explicit equation:

\begin{theo}
{\bf (Newlander-Nirenberg)} An almost complex structure $J$ yields the structure of a complex manifold if and only if
it is integrable, which means $ \bar{\partial} ^2 = 0. $
\end{theo}

Obviously the group of oriented diffeomorphisms of $M$ acts on the space of complex structures, hence one can
define in few words some basic concepts.

\begin{defin}
Let $\sD iff ^+ (M)$ be the group of orientation preserving diffeomorphisms of $M$ , and let $\sC (M)$ the space of complex structures
on $M$. Let $\sD iff ^0 (M) \subset \sD iff ^+ (M)$ be the connected component of the identity, 
the so called subgroup of diffeomorphisms which are isotopic
to the identity.

Then Dehn  (\cite{dehn}) defined the mapping class group of $M$ as 
$$\sM ap (M) : =  \sD iff ^+ (M) /  \sD iff ^0 (M),$$
while the Teichm\"uller space of $M$, respectively the moduli space of complex structures on $M$  are defined as
$$ \sT (M) : = \sC (M) / \sD iff ^0 (M) , \ \mathfrak M (M)  : = \sC (M) / \sD iff ^+ (M).$$
\end{defin}

These definitions are very clear,  however they only show that these objects are topological spaces,
and that 

$$ (*) \   \mathfrak M (M) =  \sT (M)  /  \sM ap (M) .$$

The simplest examples here are two: complex tori and compact complex curves.

The  example of complex tori sheds light on the important  question concerning  the determination of the connected
components of $\sC (M)$, which are called the
deformation classes
in the large of  the complex structures on $M$ (cf.
\cite{cat02}, \cite{cat04}).

Complex tori are 
 parametrized  by an open set $\mathcal T_n$ of the complex
Grassmann Manifold $Gr(n,2n)$, image of the open set of matrices

$\{ \Omega \in Mat(2n,n; \CC) \ | \ (i)^n det  (\Omega \overline
{\Omega}) > 0 \}.$

This parametrization is very explicit: if we consider
a fixed lattice $ \Ga \cong \ZZ^{2n}$,
to each matrix $ \Omega $ as above we associate the subspace 
$$ V = ( \Omega )
(\CC^{n}),$$ so that
$ V \in Gr(n,2n)$ and $\Ga \otimes \CC \cong V \oplus \bar{V}.$

Finally, to $ \Omega $ we associate the torus $Y_V : = V / p_V (\Ga)$,
$p_V : V \oplus
\bar{V} \ra V$ being the projection onto the first addendum.

Not only we obtain in this way a connected open set
inducing  all the small deformations (cf.
\cite{k-m71}), but indeed, as  it was shown in \cite{cat02} (cf. also
\cite{cat04}) $\mathcal T_n$ is
a connected component of  Teichm\"uller space (as the letter $\mathcal T$ suggests).

It was observed however by Kodaira and Spencer already in their first article
(\cite {k-s58}, and volume II of Kodaira's
collected works)
that for $n \geq 2$  the mapping class group $ SL ( 2n, \ZZ)  $ does not act properly
discontinuously
on $\mathcal T_n$. More precisely, they show that for every non empty
open set $U \subset \mathcal T_n$ there is a point $t$ such that
the orbit  $ SL ( 2n, \ZZ)  \cdot  t$ intersects $U$ in an infinite set.

This shows that the quotient is not Hausdorff at each point, probably it is not even
a non separated  complex space.

Hence the moral is that 
for compact complex manifolds
it is better to consider, rather than the Moduli space, the
Teichm\"uller space.

Moreover, after some initial constructions
by Blanchard and Calabi (cf. \cite{bla53},  \cite{bla54}, , \cite{bla56},
\cite{cal58}) of non K\"ahler
complex structures
$X$ on manifolds diffeomorphic to a product $C \times T$, where $C$ is a
compact complex curve and $T$ is a 2-dimensional complex torus,
Sommese generalized their constructions, obtaining  (\cite{somm75}) that  the space of complex
structures on a  six dimensional real torus is not connected.

These examples were then generalized in
\cite{cat02} \cite{cat04} under the name of {\bf Blanchard-Calabi manifolds}
showing (corollary 7.8 of \cite{cat04}) that also
the space of complex structures on the product of a curve $C$
of genus $g \geq 2$ with a four
dimensional real torus is not connected, and that there is no upper bound for the dimension of
Teichm\"uller space (even when $M$ is fixed).

\bigskip

The case of compact complex curves $C$ is instead the one  which was originally 
considered by Teichm\"uller.

In this case,  if the genus $g$ is at least $2$, the
Teichm\"uller space $\sT_g$  is a bounded domain,   diffeomorphic to a ball,
contained in the vector space of quadratic differentials 
$H^0 ( C, \hol_C ( 2 K_C))$ on a fixed such curve $C$.

In fact, for each other complex structure on the oriented 2-manifold $M$ underlying 
$C$ we obtain a complex curve $C'$, and there is a unique extremal quasi-conformal
map $ f : C \ra C'$, i.e., a map such that the Beltrami distortion  $\mu_f : = \bar{\partial} f / \partial f$ 
has minimal norm (see for instance  \cite{hubbard} or \cite{ar-cor}). 

The fact that  the Teichm\"uller space $\sT_g$  is homeomorphic to a ball (see \cite{Tro} for a simple proof)
is responsible for the fact that the moduli space of curves $\mathfrak M_g$ is close to be a classifying 
space for the mapping class group (see \cite{mumshaf} and the articles by Edidin and Wahl in this Handbook).

\subsection{Kuranishi space}

Interpreting  the Beltrami distortion as a closed $(0,1)$- form with values in the dual  $(TC^{1,0})$ of the
cotangent bundle $(TC^{1,0})^{\vee}$, we obtain a particular case of the Kodaira-Spencer-Kuranishi
theory of local deformations.

In fact, by Dolbeault 's theorem, such a closed form determines a cohomology
class  in $H^1 ( \Theta_C)$, where $ \Theta_C$ is the sheaf of holomorphic 
sections of the holomorphic tangent bundle  $(TC^{1,0})$: these cohomology classes
are interpreted, in the  Kodaira-Spencer-Kuranishi theory, as infinitesimal deformations (or derivatives of a family of deformations) of a complex structure: let us review briefly how.

Local deformation theory addresses precisely  the study of the {\bf small
deformations} of a complex manifold $ Y = (M, J_0)$.

We shall use here unambiguously the double notation $TM^{0,1} = TY^{0,1} ,\  TM^{1,0} = TY^{1,0} $ to refer  to the 
splitting determined by the complex structure $J_0$. 

$J_0$ is a point in $\sC (M)$, and  a neighbourhood in the space of almost complex structures
corresponds to a distribution of subspaces which are globally defined as  graphs of an endomorphism
$$ \phi :  TM^{0,1} \ra TM^{1,0},$$
called a {\bf small variation of complex structure}, since one then defines 
 $$TM^{0,1}_{\phi} : = \{ (u, \phi (u))| \ u \in TM^{0,1} \} \subset TM^{0,1} \oplus  TM^{1,0}.$$
 
 In terms of the new $\bar{\partial}$ operator, the new one is simply obtained by considering
 $$\bar{\partial}_{\phi} : =  \bar{\partial} + \phi, $$
 and the integrability condition is given by the Maurer-Cartan equation
 
$$ (MC) \ \ \bar{\partial} (\phi) + \frac{1}{2} [ \phi, \phi ] = 0,$$
where $[ \phi, \phi ] $ denotes the Schouten  bracket, which is the composition of  exterior product of forms followed by Lie bracket of vector fields,
and  which is graded commutative.

Observe for later use that the form $  F(\phi) : = ( \bar{\partial} (\phi) + \frac{1}{2} [ \phi, \phi ] )$ is $\bar{\partial} $ closed,
if  $ \bar{\partial} (\phi)=0 $, since  then
$$ \bar{\partial}  F(\phi)  =  \frac{1}{2}  \bar{\partial} [ \phi, \phi ] =  \frac{1}{2}( [ \bar{\partial}  \phi, \phi ] + [ \phi, \bar{\partial}  \phi ])= 0.$$

Recall also the theorem of Dolbeault: if $\Theta_Y$  is the sheaf of holomorphic sections of $TM^{1,0}$, then $H^j (\Theta_Y)$
is isomorphic to the quotient space $\frac{Ker(\bar{\partial} ) }{Im (\bar{\partial} )}$ of the space of  $\bar{\partial} $ closed  $(0,j)$-forms with values in $TM^{1,0}$
modulo the space of $\bar{\partial} $-exact  $(0,j)$-forms with values in $TM^{1,0}$.

Our  $F$ is a map of degree 2 between two infinite dimensional spaces, the space of $(0,1)$-forms with values in the bundle $TM^{1,0}$,
and the space of    $(0,2)$-forms with values in $TM^{1,0}$.

Observe that, since our original complex structure $J_0$ corresponds to $\phi = 0$, the derivative $DF$
of the above equation $F$ at $\phi = 0$ is simply
$$ \bar{\partial} (\phi)= 0,$$
hence the tangent space to the space of complex structures consists of the space of $ \bar{\partial} $-closed
forms of type $(0,1)$ and with values in the bundle $ TM^{1,0}$. Moreover the derivative of $F$ surjects onto
the space of $\bar{\partial} $-exact  $(0,2)$-forms with values in $TM^{1,0}$.

We are  now going to show why we can restrict our consideration only to the class of such forms $\phi$ in the
Dolbeault cohomology group  $$H^1 (\Theta_Y): =  ker ( \bar{\partial} ) / Im ( \bar{\partial})  .$$

This is done by answering the question: how does the group of diffeomorphisms act on an almost complex structure $J$?

This is in general difficult to specify, but we can consider the infinitesimal action of a 1-parameter group
of diffeomorphisms $$\{ \psi_t : = exp ( t (\theta + \bar{\theta} ) | t \in \RR \},$$ corresponding to a differentiable vector field
$\theta$ with values in $ TM^{1,0}$ ; from now on, we shall assume
that $M$ is compact, hence the diffeomorphism $\psi_t$ is defined $\forall t \in \RR$.

We refer to \cite{kur3} and  \cite{huy}, lemma 6.1.4 , page 260, for the following calculation of the Lie derivative:

\begin{lemma}
Given a 1-parameter group
of diffeomorphisms

 $\{ \psi_t : = exp ( t (\theta + \bar{\theta} ) | t \in \RR \}$,
$(\frac{d}{dt})_{t=0} (\psi_t ^* (J_0)) $ corresponds to the small variation $ \bar{\partial} (\theta)$.
\end{lemma}

The lemma says, roughly speaking, that locally at each point $J$ the orbit for the group of diffeomorphisms
in $\sD iff^0(M)$  contains a submanifold, having as tangent space the forms in the same  Dolbeault cohomology 
class of $0$, which has  finite codimension inside   another submanifold
with tangent space the space of  $ \bar{\partial} $-closed
forms $\phi$. Hence the tangent space to the orbit space is the space of such Dolbeault cohomology classes.

Even if we `heuristically' assume  $ \bar{\partial} (\phi)= 0,$ it looks like we are still left with another equation with values
in an infinite dimensional space. However, the derivative $DF$ surjects onto the space of exact forms, while the restriction
of $F$ to the subspace  of $ \bar{\partial}$-closed forms ($\{ \bar{\partial} (\phi)= 0\}$ takes values in
the space of $ \bar{\partial} $-closed forms: this is the moral reason why indeed one can reduce the above equation $ F=0$, associated to a map
between infinite dimensional spaces, to an equation $ k=0$ for a map $k : H^1 (\Theta_Y) \ra H^2 (\Theta_Y)$,
called the Kuranishi map.

This is done explicitly via a miracolous equation (see \cite{k-m71}, \cite{Kodbook},\cite{kur4} and  \cite{montecatini} for details)  set up by Kuranishi
in order to reduce the problem to a finite dimensional one (here Kuranishi,  see \cite{kur3}, uses 
the Sobolev r- norm in order to be able to use the implicit function theorem for Banach spaces).

Here is how the Kuranishi equation is set up.

Let $\eta_1, \dots , \eta_m \in H^1 (\Theta_Y)$ be a basis for the space of harmonic (0,1)-forms with values in 
$TM^{1,0}$, and set $t : =(t_1, \dots, t_m) \in \CC^m$, so that  $ t \mapsto \sum_i t_i \eta_i$ 
establishes an isomorphism $\CC^m \cong H^1 (\Theta_Y)$. 

Then the {\em Kuranishi slice} (see \cite{palais} for a general theory of slices) is obtained
by associating to $t$ the unique power series solution of the following equation:
$$ \phi (t) =   \sum_i t_i \eta_i + \frac{1}{2} \bar{\partial}^* G [ \phi (t), \phi (t)] , $$
satisfying moreover  $\phi (t) =   \sum_i t_i \eta_i + $ higher order terms
($G$ denotes here  the Green operator).  

The upshot is that for these forms
the integrability equation simplifies drastically; the result is summarized in the following definition.

\begin{defin}
The  Kuranishi space $\frak B (Y)$ is defined as the germ of complex subspace of $H^1 (\Theta_Y)$
defined by
$\{ t \in \CC^m | \  H  [ \phi (t), \phi (t)] = 0 \} $, where $H$ is the harmonic projector onto the space 
$ H^2 (\Theta_Y)$ of harmonic forms of type $(0,2)$ and with values in $TM^{1,0}$.

 Kuranishi space $\frak B (Y)$ parametrizes  exactly the set of small variations of complex structure $ \phi (t)$
which are integrable. Hence over $\frak B (Y)$ we have a family of complex
structures which deform the complex structure of $Y$.

\end{defin}

It follows from the above arguments that the  Kuranishi space $\frak B (Y)$
surjects onto the germ of
the  Teichm\"uller space at the point corresponding to the  given
complex structure $Y = (M, J_0)$.

It fails badly to be a homeomorphism, and my favourite example for this is (see \cite{catrav}) the
one of the Segre ruled surfaces $\FF_n$,  obtained as the blow up at the origin of 
the projective cone over a rational normal curve of degree $n$, and
 described  by Hirzebruch biregularly as $\PP ( \hol_{\PP^1} \oplus \hol_{\PP^1} (n)), n \geq 0.$

Kuranishi space is here the vector space 
$$  H^1 (\Theta_{\FF_n}) \cong {\rm Ext}^1( \hol_{\PP^1}(n),  \hol_{\PP^1} )$$
parametrizing projectivizations $\PP (E)$, where the rank 2 bundle $E$
occurs  as an extension
$$ 0 \ra  \hol_{\PP^1} \ra E \ra  \hol_{\PP^1} (n)\ra 0. $$

By Grothendieck's theorem, however, $E$ is a direct sum of two line bundles,
hence we get as a possible surface only a surface $\FF_{n-2k}$,
for each $ k \leq \frac{n}{2}$. Indeed Teichm\"uller space, in a neighbourhood of the
point corresponding to $\FF_n$ consists just of a finite number of points
corresponding to each  $\FF_{n-2k}$, and where  $\FF_{n-2k}$ is in the closure
of  $\FF_{n-2h}$ if and only if $k \leq h$.

The reason for this phenomenon is the following. Recall that the form
$\phi$ can be infinitesimally changed by adding   $ \bar{\partial} (\theta)$;
now, for  $\phi = 0$, nothing is changed if $ \bar{\partial} (\theta)=0$.
i.e., if $\theta \in H^0 (\Theta_Y)$ is a holomorphic vector field. 
But the exponentials of these vector fields, which are holomorphic on
$ Y = \FF_n$, but not necessarily for $\FF_{n-2k}$, act transitively
on each   stratum of the stratification of ${\rm Ext}^1( \hol_{\PP^1}(n),  \hol_{\PP^1} )$ given by isomorphism type (each stratum is thus the set of 
 surfaces isomorphic to $\FF_{n-2k}$).

In other words, the jumping of the dimension of $H^0 (\Theta_{Y_t})$
for $t \in \frak B (Y)$ is responsible for the phenomenon.

Indeed Kuranishi, improving on a result of Wavrik (\cite{Wav}) obtained 
in \cite{kur3} the following result.

\begin{theo} {\bf( Kuranishi's third theorem)}\label{kur3}
Assume that   the dimension of $H^0 (\Theta_{Y_t})$
for $t \in \frak B (Y)$ is a constant function  in a neighbourhood of $0$.

Then there is $k > > 0$ and a neighbourhood $\frak U$ of the identity map
in the group $\sD iff (M)$, with respect to the k-th Sobolev norm,
and a neighbourhood $U$ of $0$ in $ \frak B (Y)$ 
such that, for each  $f \in \frak U$, and $t \neq t' \in U$, 
$f$ cannot yield a holomorphic map between $Y_t$ and $Y_{t'}$.
\end {theo}
 
Kuranishi's theorem (\cite{kur1},\cite{kur2}) shows that 
Teichm\"uller space can be viewed as being locally dominated by a 
complex space  of locally finite dimension
(its dimension, as we already observed,  may however be unbounded, cf.
cor. 7.7 of
\cite{cat04}).

A first consequence is that Teichm\"uller space  is locally connected by holomorphic arcs,
hence the determination of the connected components of $\sC (M)$, respectively 
 of $\sT (M)$, can be done using the original definition of 
deformation equivalence, given by Kodaira and Spencer in \cite{k-s58}.

\begin{cor}
Let $Y = (M,J)$, $Y' = (M,J')$ be two different complex structures on $M$.

Define  deformation equivalence as the equivalence relation generated by direct
deformation equivalence, where $Y$, $Y'$ are said to be {\bf direct disk deformation equivalent} if
and only if there is a proper holomorphic submersion with connected fibres $f \colon \sY \ra \Delta$,
where $\sY$ is a complex manifold, $\De \subset \CC$ is the unit disk,
and moreover there are two fibres of $f$ biholomorphic to $Y$, respectively  $ Y'$. 

Then two complex structures on $M$ yield points in the same connected component of $\sT (M)$ 
if and only if they are in the same deformation equivalence class.
\end{cor}

In the next subsections we shall illustrate the meaning of the condition that the vector spaces $H^0 (\Theta_{Y_t})$
have locally constant dimension, in terms of deformation theory.
Moreover, we shall give criteria implying that Kuranishi and Teichm\"uller space do locally coincide.

\subsection{Deformation theory and how  it is used}

One can define deformations not only for complex manifolds, but also for complex spaces.
 The technical assumption of flatness replaces then the condition that $\pi$ be a submersion.

\begin{defin}
1) A {\bf deformation}  of a compact complex space $X$ is a pair
consisting of

1.1) a flat proper morphism $ \pi : \X \ra T$ between connected complex spaces (i.e.,
$\pi^* : \hol_{T,t} \ra \hol_{\X,x}$ is a flat ring extension for each
$ x$ with $ \pi (x) = t$)

1.2) an isomorphism $ \psi : X \cong \pi^{-1}(t_0) : = X_0$ of $X$ with a fibre
$X_0$ of $\pi$.

2.1) A {\bf small deformation} is the germ $ \pi : (\X, X_0) \ra (T, t_0)$ of
a deformation.

2.2) Given a deformation  $ \pi : \X \ra T$ and a morphism $ f : T' \ra T$
with $ f (t'_0) = t_0$, the {\bf pull-back} $ f^* (\X)$
is the fibre product $ \X': = \X \times_T T'$ endowed with the projection onto
the second factor $T'$ (then $ X \cong X'_0$).

3.1) A small deformation  $ \pi : \X \ra T$ is said to be {\bf versal} or {\bf complete} if every
other small deformation $ \pi : \X' \ra T'$ is obtained from it via pull back; it is said
to be {\bf semi-universal} if the differential of $ f : T' \ra T$
 at $t'_0$ is uniquely determined, and {\bf universal}  if the morphism
 $f$ is uniquely determined.

4) Two compact complex manifolds $X,Y$ are said to be {\bf direct deformation
equivalent} if there are a deformation $ \pi : \X \ra T$ of $X$ with $T$
  irreducible and where all the fibres are smooth,
and an isomorphism $ \psi' : Y \cong \pi^{-1}(t_1) : = X_1$ of $Y$ with a
fibre $X_1$ of $\pi$.
\end{defin}

Let's however come back to the case of complex manifolds, observing that in a small deformation
of a compact complex manifold one can shrink the base $T$ and assume that all the fibres are smooth.

We can now state the results of Kuranishi and Wavrik ((\cite{kur1}, \cite{kur2}, \cite{Wav}) in the language of deformation theory.

\begin{theo}
{\bf (Kuranishi).} Let $Y$
be a compact
complex manifold: then

I) the Kuranishi family   $ \pi : (\Y, Y_0) \ra (\frak B (Y), 0)$ of
$Y$ is semiuniversal.

II) $(\frak B (Y), 0)$ is unique up to (non canonical) isomorphism, and is a germ of analytic
subspace of the vector space
$ H^1 (Y, \Theta_Y)$, inverse image of the origin under a local
holomorphic map (called Kuranishi map and denoted by $k$) 
$ k : H^1 (Y, \Theta_Y) \ra H^2 (Y, \Theta_Y) $ whose differential vanishes
at the origin. 

Moreover  the  quadratic term in the Taylor development of the Kuranishi map
$k$ is  given by  the bilinear map $ H^1 (Y, \Theta_Y) \times
   H^1 (Y, \Theta_Y) \ra  H^2 (Y, \Theta_Y)$, called Schouten bracket,
   which is the composition of  cup product  followed by Lie bracket of vector fields.

III) The Kuranishi family is a versal deformation of $Y_t$ for $t \in  \frak B (Y)$.

IV) The Kuranishi family is universal if $H^0 (Y, \Theta_Y)= 0.$

V) {\bf (Wavrik)} The Kuranishi family is universal if $ \frak B (Y)$ is reduced and $h^0 (Y_t, \Theta_{Y_t}) : = \rm{dim}\ H^0 (Y_t, \Theta_{Y_t}) $
is constant  for $t \in   \frak B (Y)$ in a suitable neighbourhood of $0$.
\end{theo}

In fact Wavrik in his article (\cite{Wav}) gives a  more general result than V); as pointed out by a referee, the
same criterion has also been proven by Schlessinger (prop. 3.10 of \cite{FAR}).

Wavrik says that the Kuranishi space
is a local moduli space under  the assumption that $h^0 (Y_t, \Theta_{Y_t})$
is locally constant. This terminology can however be confusing, as we shall show,
since in no way the Kuranishi space is like the moduli space locally, 
even if one divides out by the action of the group  $ Aut (Y)$ of biholomorphisms of
$Y$.

The first  most concrete question is how one can calculate the Kuranishi space and the Kuranishi family.
In this regard, the first resource is to try to use the implicit functions theorem.

For this purpose one needs to calculate the Kodaira Spencer map of a family $ \pi : (\Y, Y_0) \ra (T, t_0)$
of complex manifolds having a smooth base $T$.
This is defined as follows: consider the cotangent bundle sequence of the fibration
$$ 0 \ra \pi^* (\Omega^1_T) \ra  \Omega^1_{\Y} \ra  \Omega^1_{\Y| T} \ra 0,$$
and the direct image sequence of the  dual sequence of bundles,
$$ 0 \ra \pi_* (\Theta_{\Y| T}) \ra   \pi_* (\Theta_{\Y}) \ra   \Theta_T  \ra  \sR^1  \pi_* (\Theta_{\Y| T}) .$$

Evaluation at the point $t_0$ yields a map $\rho$ of the tangent space to $T$ at $t_0$ into $ H^1 (Y_0, \Theta_{Y_0})$,
which is the derivative of the variation of complex structure (see \cite{k-m71} for a more concrete description,
but beware that the  definition given above is the most effective for calculations). 

\begin{cor}
Let $Y$
be a compact
complex manifold and assume that we have a family $ \pi : (\Y, Y_0) \ra (T, t_0)$ with smooth base $T$, such that $Y \cong Y_0$,
and such that the Kodaira Spencer map $\rho_{t_0}$ surjects onto $ H^1 (Y, \Theta_Y)$.

Then the Kuranishi space $\frak B (Y)$ is smooth and there is a submanifold $T' \subset T$ which maps isomorphically to $\frak B (Y)$;
hence the Kuranishi family is the restriction of $\pi$ to $T'$. 
\end{cor}

The key point is that, by versality  of the Kuranishi family, there is a morphism $f : T \ra \frak B (Y)$ inducing $\pi$ as a pull back,
and $\rho$ is the derivative of $f$. 

This approach clearly works only if $Y$ is {\bf unobstructed}, which simply means
that $\frak B (Y)$ is smooth. In general it is difficult to describe the Kuranishi map, and even calculating the quadratic term
is nontrivial (see \cite{quintics} for an interesting  example).

In general, even if it is difficult to calculate the Kuranishi map,  Kuranishi theory  gives a lower bound for the `number of moduli'
of $Y$, since it shows that $\frak B (Y)$ has dimension  $\geq  h^1 (Y, \Theta_Y) - h^2 (Y, \Theta_Y)$. In the case of curves
$ H^2 (Y, \Theta_Y) = 0$, hence curves are unobstructed; in the case of a surface $S$ 
$$ \rm{dim} \frak B (S ) \geq  h^1 ( \Theta_S) - h^2 ( \Theta_S) = - \chi ( \Theta_S) + h^0 ( \Theta_S) = 10 \chi (\hol_S) - 2 K^2_S + h^0 ( \Theta_S).$$

The above is the Enriques' inequality (\cite{enr}, observe that Max Noether postulated equality),  proved by Kuranishi in all cases and also for non algebraic surfaces.

There have been recently two examples where resorting
to the Kuranishi theorem in the obstructed case has been useful.

The first one appeared  in a preprint by  Clemens (\cite{clemens1}), who then published the  proof in \cite{clemens}; it shows that if a manifold is K\"ahlerian, then there are fewer obstructions than foreseen,
since a small deformation $Y_t$ of a K\"ahler manifold is again K\"ahler, hence the Hodge decomposition still holds
for $Y_t$. 

Another independent proof was  given by Manetti  in \cite{Manobs}.

\begin{theo}{\bf (Clemens-Manetti)}\label{Hodgekills}
Let $Y$ be a compact complex K\"ahler manifold.

 Then there exists an analytic automorphism of  $ H^2 (Y, \Theta_Y)$ with linear part equal to the identity,
 such that the Kuranishi map $ k : H^1 (Y, \Theta_Y) \ra H^2 (Y, \Theta_Y) $
takes  indeed values in the intersection of the subspaces 
$$ Ker  ( H^2 (Y, \Theta_Y)  \ra  Hom ( H^q (\Omega^p_Y) , H^{q+2} (\Omega^{p-1}_Y))$$
(the linear map is induced by cohomology cup product and tensor contraction). 
\end{theo}

Clemens' proof uses directly the Kuranishi equation, and a similar method was used by S\"onke Rollenske
in \cite{rol1}, \cite{rol2}  in order to consider the deformation theory of complex manifolds yielding 
left invariant complex structures on nilmanifolds. Rollenske  proved, among other results,  the following 

\begin{theo}{\bf (Rollenske)}
Let $Y$ be a compact complex manifold corresponding to a left invariant complex structure
on a real nilmanifold. Assume that the following condition is verified:

(*) the inclusion of the complex  of left invariant forms of pure antiholomorphic type in the 
Dolbeault complex $$(\bigoplus  _p H^0(\sA ^{(0,p)}(Y)),\overline{ \partial})$$ yields an isomorphism of cohomology groups.

 Then every small deformation of the complex structure of $Y$ consists of left invariant complex structures.
\end{theo} 

The main idea, in spite of the technical complications, is to look at Kuranishi's equation, and to see that everything is then left invariant.

Rollenske went over in \cite{rol3} and showed that for the complex structures on nilmanifolds which are complex parallelizable
Kuranishi space is defined by explicit polynomial equations, and most of the time singular.

There have been several attempts to have a more direct approach to the understanding of the Kuranishi map,
namely to do things more algebraically and giving up to consider the Kuranishi slice.
This approach has been pursued for instance in \cite{SS} and effectively applied by Manetti.
For instance, as already mentioned, Manetti (\cite{Manobs}) gave a nice elegant proof of the above theorem \ref{Hodgekills}
using the notion of  differential graded Lie algebras, abbreviated by the acronym DGLA 's.

The typical example of such a DGLA is provided by the Dolbeault complex 
$$(\bigoplus _p H^0 ( \sA ^{(0,p)}(TM^{1,0}_Y)),\overline{ \partial})$$
further endowed with the operation of Schouten  bracket (here: the composition of  exterior product followed by Lie bracket of vector fields), which is graded commutative.

The main thrust is to look at solutions of the Maurer Cartan equation
$ \bar{\partial} (\phi) + \frac{1}{2} [ \phi, \phi ] = 0$ modulo gauge transformations, i.e., exponentials of sections
in $H^0 ( \sA ^{(0,0)}(TM^{1,0}_Y))$.

The deformation theory concepts generalize from the case of deformations of compact complex manifolds
to the more general setting of DGLA's , which seem to govern almost all of the deformation type
problems (see for instance  \cite{man09}).

\subsection{Kuranishi and Teichm\"uller}

Returning to our setting where we considered the  closed subspace $ \sC (M)$  of $\sA \sC (M)$ consisting of the set of complex structures
on $M$, it is clear that there is a universal tautological family of complex structures
parametrized by $ \sC (M)$, and with total space
$$\mathfrak U_{ \sC (M)} : =  M   \times  \sC (M) ,$$
on which the group   $\sD iff ^+ (M)$ naturally acts, in particular  $\sD iff ^0(M)$.

A rather simple observation is that $\sD iff ^0(M)$ acts freely on $ \sC (M)$ if and only if for each complex structure
$Y$ on $M$ the group of biholomorphisms $Aut(Y)$ contains no automorphism which is differentiably isotopic
to the identity (other than the identity).

\begin{defin}
A compact complex manifold $Y$  is said to be {\bf rigidified}  if $Aut(Y) \cap \sD iff ^0(Y) = \{ Id_Y \}.$

A compact complex manifold $Y$  is said to be cohomologically rigidified if $Aut(Y) \ra Aut (H^* (Y, \ZZ))$
is injective, and rationally cohomologically rigidified if  $Aut(Y) \ra Aut (H^* (Y, \QQ))$ is injective.
\end{defin}

The condition of being rigidified is obviously stronger than the condition $H^0 (\Theta_Y)=0$, which is necessary, else there is
a positive dimensional Lie group of biholomorphic self maps, and is weaker  than the condition of being cohomologically rigidified.

Compact curves of genus $ g \geq 2$ are rationally cohomologically rigidified since if $\tau : C \ra C$ is an automorphism
acting trivially on cohomology, then in the product $ C \times C$ the intersection number of the diagonal $\De_C$ with
the graph $\Ga_{\tau}$  equals the self intersection of the diagonal, which is the Euler number $ e(C) = 2 - 2g < 0$.
But, if $\tau$ is not the identity, $\Ga_{\tau}$ and $\De_C$ are irreducible and distinct, and their intersection number
is a non negative number, equal to the number of fixed points of $\tau$, counted with multiplicity: a contradiction. 

It is an interesting question whether compact complex manifolds of general type are rigidified.
It is known that already for  surfaces of general type there are examples which are not rationally cohomologically rigidified
(see a partial classification done by Jin Xing Cai in \cite{cai}), while examples which are not cohomologically
rigidified might exist among surfaces isogenous to a product (potential candidates have been proposed by Wenfei Liu).

Jin Xing Cai pointed out to us that, for simply connected (compact) surfaces, by a result of Quinn (\cite{quinn}), every automorphism acting trivially in rational cohomology is isotopic to the identity, and that he conjectures that
  simply connected  surfaces of general type are rigidified (equivalently, rationally cohomologically rigidified).

\begin{rem}
Assume that the complex manifold $Y$ has $H^0 (\Theta_Y)=0$, or satisfies Wavrik's condition, but is not rigidified: then by Kuranishi' s  third 
theorem, there is an automorphism $ f \in Aut (Y) \cap \sD iff ^0(Y)$ which lies outside of a fixed neighbourhood of the identity. $f$ acts therefore on the Kuranishi space, hence, in order that the natural map from Kuranishi space to Teichm\"uller space be injective, $f$ must act trivially on $ \frak B (Y)$, which means that $f$ remains biholomorphic for
all small deformations of $Y$.

\end{rem}

At any case, the condition of being rigidified implies that the tautological family of complex structures
descends to a universal family of complex structures on Teichm\"uller space:

$$\mathfrak U_{ \sT (M)} : =   (M   \times  \sC (M)) / \sD iff ^0(M) \ra  \sC (M)) / \sD iff ^0(M) = \sT (M) .$$
on which the mapping class group acts.

Fix now a complex structure yielding a compact complex manifold $Y$, and compare with the Kuranishi family
$$\sY \ra \frak B (Y).$$

Now, we already remarked that there is a locally surjective continuous map of $ \frak B (Y)$ to the germ
 $\sT (M)_Y$ of $\sT (M)$ at the point corresponding to the complex structure yielding $Y$. 
 For curves this map is a local homeomorphism, and this fact provides a complex structure on Teichm\"uller space.
 
\begin{rem}
Indeed we observe that more generally, if 

1) the Kuranishi family is universal at any point

2)   $ \frak B (Y) \ra \sT (M)_Y$ is a local homeomorphism at every point, then

 Teichm\"uller space has a natural structure of complex space.

Moreover

3) since $ \frak B (Y) \ra \sT (M)_Y$ is surjective,   it is a local homeomorphism iff it is injective; in fact, since $\sT(M)$ has the quotient topology and it is the quotient by a group action, and  $ \frak B (Y)$ is a local slice for a subgroup of
$\sD iff ^0(M)$, the projection  $ \frak B (Y) \ra \sT (M)_Y$ is open.

\end{rem} 

The simple idea used by Arbarello and Cornalba (\cite{ar-cor}) to reprove the result for curves is to establish the universality
of the Kuranishi family for continuous families of complex structures.

In fact, if any family is locally induced by the Kuranishi family, and we have rigidified manifolds only,
then there is a continuous inverse to the map $ \frak B (Y) \ra \sT (M)_Y$, and we have the desired local homeomorphism
between Kuranishi space and Teichm\"uller space.

Since there are many cases (for instance, complex tori) where Kuranishi and Teichm\"uller space coincide,
yet the manifolds are not rigidified, we give a simple criterion.

\begin{prop}\label{kur=teich}
 1)  The continuous map $\pi \colon  \frak B (Y) \ra \sT (M)_Y$ is a local homeomorphism
  between Kuranishi space and Teichm\"uller space if there is an
  injective continuous map
  $f \colon  \frak B (Y) \ra Z$, where $Z$ is Hausdorff, which factors through $\pi$.
  
 2)  Assume that $Y$ is a compact K\"ahler manifold and that the local period map $f$ is injective:
  then $\pi \colon  \frak B (Y) \ra \sT (M)_Y$ is a local homeomorphism.
  
  3) In particular, this holds if $Y$ is K\"ahler with trivial canonical divisor \footnote{As observed by a referee, the same proof works when 
   $Y$ is K\"ahler with torsion canonical divisor, since one can consider the local period map of the canonical cover of $Y$}.

\end{prop}
 
 \Proof
1) : observe that, since $ \frak B (Y)$ is locally compact and $Z$ is Hausdorff,
 it follows that $f$ is a homeomorphism with its image $ Z' : = Im f \subset Z$.
  Given the factorization $ f = F \circ \pi$, then the inverse of $\pi$ is the composition
  $ f^{-1} \circ F$, hence $\pi$ is a homeomorphism.
  
 2) :  if $Y$ is K\"ahler, then every small deformation $Y_t$ of $Y$ is still K\"ahler, as it is well known (see \cite{k-m71}).
  
  Therefore one has the Hodge decomposition 
  $$ H^* (M, \CC) =   H^* (Y_t, \CC) = \bigoplus_{p,q} H^{p,q} (Y_t)$$
  and the corresponding period map  $f \colon  \frak B (Y) \ra \mathfrak D$, 
  where $ \mathfrak D$ is the period domain classifying Hodge structures
  of type $ \{ (h_{p,q}) |0 \leq p,q,  p + q \leq 2n \}.$ 
  
  As shown by Griffiths in \cite{griff1}, see also \cite{griff2} and \cite{vois}, the period map is indeed holomorphic, in particular continuous,
  and  $ \mathfrak D$ is a separated  complex manifold, hence 1) applies.

 3) the previous criterion applies in several situations, for instance, when $Y$ is a
 compact K\"ahler manifold with trivial canonical bundle. 
 
 In this case the Kuranishi space
 is smooth (this is the so called Bogomolov-Tian-Todorov theorem, compare \cite{bogomolov}, \cite{tian}, \cite{todorov}, and see also \cite{ran} 
 and \cite{kawamata} for more general results) and the local period map for the period of holomorphic n-forms
 is an embedding, since the derivative of the period map, according to \cite{griff1} is given
 by cup product
 $$\mu \colon  H^1 (Y, \Theta_Y)  \ra   \oplus_{p,q} Hom ( H^q (\Omega^p_Y) , H^{q+1} (\Omega^{p-1}_Y))$$
 $$ = \oplus_{p,q} Hom (H^{p,q}(Y) , H^{p-1,q+1}(Y)) .$$
 
 If we apply it for $q=0, p=n$, we get  that $\mu$ is injective, since  by Serre duality
 $  H^1 (Y, \Theta_Y)  =  H^{n-1} (Y, \Omega^1_Y \otimes \Omega^n_Y)^{\vee}$ 
 and cup product with $H^0 (\Omega^n_Y)$ yields an isomorphism with
 $H^{n-1} (Y, \Omega^1_Y )^{\vee}$ which is  by Serre duality exactly isomorphic to 
  $ H^{1} (\Omega^{n-1}_Y)$.

  \qed
  
 As we shall see later, a similar criterion applies to show `Kuranishi= Teichm\"uller' for most minimal models of surfaces
 of general type.

 For more general complex manifolds, such that the Wavrik condition holds, then the Kuranishi family is universal at any point, so a program which has been in the air for a quite long time has been the one to glue together
 these Kuranishi families, by a sort of analytic continuation giving another variant of Teichm\"uller space.
 
 We hope to be able to return on this point in the future.

\section{The role of singularities}

\subsection{Deformation of singularities and singular spaces}

The basic analytic result is the generalization due to Grauert of
Kuranishi's theorem (\cite{grauert}, see also \cite{sernesi} for the algebraic analogue)

\begin{theo}
{\bf Grauert's Kuranishi type theorem for complex spaces.} Let $X$
be a compact
complex space: then

I) there is a semiuniversal deformation $ \pi : (\X, X_0) \ra (T, t_0)$ of
$X$, i.e., a deformation such that every other small deformation
$ \pi' : (\X', X'_0) \ra (T', t'_0)$ is the pull-back of $\pi$ for
an appropriate morphism $f : (T', t'_0) \ra (T, t_0)$ whose
differential at $t'_0$ is uniquely determined.

II) $(T, t_0)$ is unique up to isomorphism, and is a germ of analytic
subspace of the
vector space
$\T ^1$ of first order deformations. 

$(T, t_0)$ is the  inverse image of the origin under a local
holomorphic map (called Kuranishi map and denoted by $k$) 
$$ k :
\T ^1  \ra \T ^2 $$  to the finite dimensional vector space $\T ^2$ (called {\bf obstruction space}), and  whose differential vanishes
at the origin (the point corresponding to the point $t_0$).

If $X$ is reduced, or if the singularities of $X$ are local complete intersection singularities, then $\T ^1 =  {\rm Ext }^1 (\Omega^1_X, \hol_X ).$ 

If the singularities of $X$ are local complete intersection singularities, then $ \T ^2 = {\rm Ext }^2 (\Omega^1_X, \hol_X) $ .
\end{theo}

Recall once more that this result reproves the theorem of Kuranishi (\cite{kur1}, \cite{kur2}), which  dealt with the 
case of compact
complex manifolds, where  $\T ^j = {\rm Ext }^j (\Omega^1_X,
\hol_X ) \cong  H^j (X, \Theta_X)$,  $\Theta_X : = \sH om ( \Omega^1_X,
\hol_X ) $ being  the sheaf of holomorphic vector fields.

There is also the local variant, concerning isolated singularities, which was obtained by Grauert in \cite{grauert1} extending
the  earlier result by Tyurina in the unobstructed case where   $  {\sE xt }^2 (\Omega^1_X, \hol_X )_{x_0} = 0$ 
(\cite{tyurina1}).

\begin{theo}
{\bf Grauert' s theorem for deformations of isolated singularities..} Let $(X, x_0)$
be the germ of an isolated singularity of a reduced complex space:
then

I) there is a semiuniversal deformation $ \pi : (\X, X_0, x_0) \ra (\CC^n, 0) \times (T, t_0)$ of
$(X, x_0)$, i.e., a deformation such that every other small deformation
$ \pi' : (\X', X'_0, x'_0) \ra  (\CC^n, 0) \times (T', t'_0)$ is the pull-back of $\pi$ for
an appropriate morphism $f : (T', t'_0) \ra (T, t_0)$ whose
differential at $t'_0$ is uniquely determined.

II) $(T, t_0)$ is unique up to isomorphism, and is a germ of analytic
subspace of the
vector space
$\T ^1_{x_0}: = {\sE xt }^1 (\Omega^1_X, \hol_X )_{x_0},$ inverse image of the origin under a local
holomorphic map (called Kuranishi map and denoted by $k$) 
$$ k :
{\rm \T ^1_{x_0}} = {\sE xt }^1 (\Omega^1_X, \hol_X )_{x_0}  \ra \T ^2_{x_0} $$  to the finite dimensional vector space $\T ^2_{x_0}$ (called {\bf obstruction space}), and  whose differential vanishes
at the origin (the point corresponding to the point $t_0$).

The obstruction space  $ \T ^2_{x_0} $  equals  $  {\sE xt }^2 (\Omega^1_X, \hol_X )_{x_0}$ if the  singularity of $X$ is normal.
\end{theo}

For the last assertion, see \cite{sernesi}, prop. 3.1.14, page 114.

The case of complete intersection singularities was shown quite generally to be unobstructed
by Tyurina in the hypersurface case (\cite{tyurina1}), and then by Kas-Schlessinger in \cite{k-s}.

This case lends itself to a very explicit description.

Let $(X,0)\subset \CC^n$ be the complete intersection $ f^{-1} (0)$, where $$ f = (f_1, \dots , f_p) \colon (\CC^n,0) 
\ra (\CC^p,0).$$

Then the ideal sheaf $\sI _X$ of $X$ is generated by $(f_1, \dots , f_p)$ and the conormal sheaf 
$\sN ^{\vee}_X : = \sI _X / \sI _X^2 $ is locally free of
rank $p$ on $X$.

Dualizing the exact sequence 

$$ 0 \ra  \sN ^{\vee}_X \cong  \hol_X^p \ra \Omega^1_{\CC^n} \otimes \hol_X   \cong  \hol_X^n  \ra \Omega^1_X \ra 0 $$
we obtain (as $ \Theta_X : = \sH om ( \Omega^1_X, \hol_X)$)
$$ 0 \ra \Theta_X  \ra \Theta_{\CC^n} \otimes \hol_X   \cong  \hol_X^n  \ra  \sN_X \cong  \hol_X^p \ra
 \sE xt^1 ( \Omega^1_X, \hol_X) \ra  0 $$
 
 which represents $\T ^1_{0}: = {\sE xt }^1 (\Omega^1_X, \hol_X )_{0}$ as a quotient of $ \hol_{X,0}^p$,
 and as a finite dimensional vector space (whose dimension will be denoted as usual by $\tau$, which is the so called Tyurina number).

 Let $(g^1, \dots,g^{\tau}) \in  \hol_{X,0}^p$, $g^i =  (g^i_1, \dots, g^i_p) $ represent a basis of $\T ^1_{0}$.
 
 Consider now the complete intersection
 $$ (\mathfrak X , 0) : = V (F_1, \dots, F_p) \subset   (\CC^n  \times  \CC^{\tau},0)$$
 where
  $$ F_j (x,t) : = f_j (x) + \sum_{i=1}^{ \tau} t_i  g^i_j(x).$$ 
  Then
  $$ \xymatrix {  (X, 0)  \ar@{^{(}->}[r]^i  &(\mathfrak X , 0) \ar[r]^\phi & (\CC^{\tau},0)}$$
 where $i$ is the inclusion and $\phi$ is the projection,
 yields  the semiuniversal deformation of $ (X,0)$.

\smallskip

In the case $p=1$ of hypersurfaces the  above representation of $\T ^1_{0}: = {\sE xt }^1 (\Omega^1_X, \hol_X )_{0}$ as a quotient of $ \hol_{X,0}$
yields the well known formula:
$$ \T ^1_{0} =   \hol_{\CC^n,0} / (f, f_{x_1}, \dots,  f_{x_n}),$$
where $ f_{x_i} : = \frac{\partial f}{\partial x_i}.$
 
 The easiest example is then the one of an ordinary quadratic singularity, or node, where we have $p=1$, and $ f =  \sum_{i=1, \dots n} x_i^2$. 
 
 Then our module $\T ^1_{0} = \hol_{\CC^n, 0}/ (x_i)$ and the deformation is
 $$ f + t =  \sum_{i=1}^n x_i^2 + t = 0.  $$

\subsection{ Atiyah's example and three of its implications}

Around 1958 Atiyah (\cite{atiyah}) made a very important discovery concerning families of surfaces acquiring
ordinary double points. His result was later extended by Brieskorn and Tyurina (\cite{tju}, \cite{brieskorn2}, \cite{nice}) to the more general case of
rational double points, which are the rational hypersurface singularities, and which are referred to as RDP's or
as Du Val singularities (Patrick Du Val classified them as the surface singularities which do not impose adjunction conditions,
see\cite{duval}, \cite{artin}, \cite{reid1},  \cite{reid2})) or as Kleinian singularities (they are analytically isomorphic to a quotient $\CC^2 / G$,
with $ G \subset SL ( 2 , \CC)$).

The crucial  part of the story takes place at the local level, i.e., when one deforms the ordinary double point singularity
$$ X = \{ (u,v,w ) \in \CC^3  |  w^2
= uv \} .$$

In this case the semiuniversal deformation is, as we saw, the family 
$$\XX = \{ (u,v,w,t ) \in \CC^4|  w^2
- t = uv \,\}$$
mapping to $\CC$ via the projection over the variable $t$; and one observes here that $\XX \cong \CC^3$.

The minimal resolution of $X$ is obtained blowing up the origin, but we cannot put the minimal resolutions of the $X_t$
together.

One can give two reasons for this fact. The first is algebro geometrical, in that 
any normal modification of $\XX$ which is an isomorphism outside the origin,
and is such that the fibre over the origin has dimension at most 1, must be necessarily an isomorphism.

The second reason is that  the restriction of the family (of manifolds with boundary)  to the punctured disk $\{ t \neq 0 \}$ is not topologically trivial, its monodromy
being given by a Dehn twist around the vanishing two dimensional sphere (see\cite{milnor}). 

As a matter of fact the square of the Dehn twist is
differentiably isotopic  to the identity, as it is shown by the fact that  
the family $X_t$ admits
a simultaneous resolution after that we perform a base change
$$ t = \tau^2 \Rightarrow  w^2 - \tau^2 = uv.$$

\begin{defin}
Let $\XX \ra T'$ be the family where $$\XX = \{ (u,v,w,\tau ) |  w^2
- \tau^2 = uv \}$$
and $T'$ is the affine line with coordinate $\tau$.

$\XX $ has an isolated ordinary quadratic singularity which can be resolved either by blowing up
the origin (in this way we get an exceptional divisor $ \cong \PP^1 \times \PP^1$)
or by taking the closure of one of two distinct rational maps to $\PP^1$.
The two latter resolutions are called the {\bf small resolutions}.

One defines $\sS \subset \XX \times \PP^1$ to be one of the small
resolutions of $\XX$,
and $\sS'$ to be the other one, namely:

$$\sS : \{ (u,v,w,\tau)(\xi) \in  \XX \times \PP^1| \
\frac{w-\tau}{u} =   \frac{v}{w+\tau} = \xi \}$$
$$\sS' : \{ (u,v,w,\tau)(\eta) \in  \XX \times \PP^1| \
\frac{w+\tau}{u} =   \frac{v}{w-\tau} = \eta \}.$$
\end{defin}

Now, the  two families on the disk 
$\{ \tau \in \CC| | \tau| < \epsilon \}$ are clearly isomorphic by the automorphism 
 $\sigma_4$ such that $\sigma_4(u,v,w,\tau) = (u,v,w,- \tau) $,
 
 On the other hand,   the restrictions of the  two families to the punctured disk
$\{ \tau \neq 0\}$ are clearly isomorphic by the automorphism acting 
as the identity on the variables  $(u,v,w,\tau)  $, since over the punctured disk
these two families coincide with the family $\XX$.

This  automorphism yields a birational map $\iota : \sS \dasharrow \sS'$ which however does not extend biregularly,
since $\xi   u = v  \eta ^{-1}$.

The automorphism  $\sigma : = \sigma_4 \circ \iota $ acts on the restriction $\sS^*$ of the family
$\sS$ to the punctured disk, and it acts on the given differentiably trivialized
family $\sS^*$ of manifolds with boundary via the Dehn twist on the vanishing 2-sphere.

For $\tau = 0$ the Dehn twist cannot yield a holomorphic map $\phi \colon S_0 \ra S_0$,
since every biholomorphism $\phi$ sends the (-2)-curve $E$ to itself ($E$ is the only holomorphic curve in its
homology class), hence it acts on the  normal bundle of $E$ by scalar multiplication, 
therefore by an action which is homotopic to the identity in a neighbourhood of $E$:
a contradiction.

From the above observations, one can derive several `moral' consequences, when one globalizes
the procedure.

Assume now that we have a family of compact algebraic  surfaces $X_t$ such that $X_t$
is smooth for $ t \neq 0$, and, for $t =0$, it acquires a node.

We can then take the corresponding families $S_{\tau}$ and  $S'_{\tau}$ of smooth surfaces.

We can view the family $S_{\tau}$ as the image of a 1 dimensional complex disk in the Teichm\"uller
space $\sT(S_0)$ of $S_0$, and then the Dehn twist $\sigma$ yields a self map 
$$\sigma^* \colon \sT(S_0) \ra \sT(S_0).$$

It has the property that $\sigma^* (S_{\tau}) = S_{- \tau}$ for  $\tau \neq 0$, but for
$\tau = 0$, we have that  $\sigma^* (S_{0}) \neq S_0$, since  a map homotopically
equivalent to the  Dehn twist cannot yield a biholomorphic map.

Hence we get two different points of  $\sT(S_0)$, namely, $\sigma^* (S_{0}) \neq S_0$,
which are both limits  $ lim_{\tau \ra 0}  \sigma^* (S_{\tau}) =  lim_{\tau \ra 0}  S_{- \tau}$
and the conclusion is the following theorem, which is a slightly different version of  a  result of
Burns and Rapoport (\cite{b-r}).

\begin{theo}\label{nonseparated}
Let $S_0$ be a compact complex surface which contains a (-2)-curve $E$, i.e., a smooth rational curve
with self intersection equal to $-2$, obtained from the resolution of a normal surface
$X_0$ with exactly one  singular point , which is an ordinary quadratic  singularity.

Assume further that $X_0$ admits a smoothing deformation.  

Then the Teichm\"uller space   $\sT(S_0)$ is not separated.

\end{theo}

That such a surface exists is obvious: it suffices, for each degree $ d \geq 2$, to consider a surface $X_0$
in $ \PP^3$, with equation $ f_0 (x_0, x_1,x_2,x_3) = 0$, and such that there is no monomial divisible by $x_0^{d-1}$ 
appearing in $f_0$ with non zero coefficient.
The required smoothing is gotten by setting $X_t : = \{ f_t : = f_0 + t x_0^d = 0 \}$.

This example can of course be interpreted in a second way, and with a completely different wording (non separatedness of some 
Artin moduli stack), which I will try to briefly explain in a concrete way.

It is clear that  $\sigma^* (S_{0}) \neq S_0$ in Teichm\"uller space, but  $\sigma^* (S_{0})$ and  $ S_0$
yield the same point in the moduli space. 

 Think of the family $S_{\tau}$ as  a 1 dimensional complex disk in the Kuranishi space of $S_0$: then
 when we map this disk to the moduli space we have two isomorphic surfaces, namely, since
  $\sigma^* (S_{\tau}) = S_{- \tau}$ for  $\tau \neq 0$, we identify the point $\tau$ with the point $-\tau$.
  
If we consider a disk $\De $, then we get an equivalence relation in $\De \times \De$ which identifies
 $\tau$ with the point $-\tau$. We do not need to say that $\tau = 0$ is equivalent to itself,
 because this is self evident. However, we have seen that we cannot extend the self map $\sigma$
 of the family $\sS^*$ to the full family $\sS$. Therefore, if we require that equivalences come from families,
 or, in other words, when we glue Kuranishi families, we obtain the following.
 
 The equivalence relation in  $\De \times \De$ is the image of  two complex curves, one being the disk
 $\De$, the other being the punctured disk $\De^*$.
 
  $\De$ maps to the diagonal $\De \times \De$, i.e., $ \tau \mapsto (\tau, \tau)$, while
  the punctured disk $\De^*$ maps to the antidiagonal, deprived of the origin,
  that is,$ \tau \neq 0, \tau  \mapsto (\tau, - \tau)$. 
 
 The quotient  in the category of complex spaces is indifferent to the fact that we cannot have a family
 extending the isomorphism $\iota$ given previously across $\tau = 0$, and the quotient is the disk $\De_t$
 with coordinate $ t : = \tau^2$.
 
 But over the disk $\De_t$ there will not be, as already remarked, a family of smooth surfaces.
 
 This example by Atiyah motivated Artin in \cite{artinstacks} to introduce his theory of Artin stacks,
 where one takes quotients by maps which are \'etale on both factors, but not proper
 ( as the map of  $\De^*$ into  $\De \times \De$).
 
A third implication of Atiyah's example will show up in the section on automorphisms.

 \bigskip

\section{Moduli spaces for surfaces of general type}

\subsection{Canonical models of surfaces of general type.}

In the birational class of a non ruled surface there is, by the theorem of Castelnuovo (see e.g. \cite{arcata}),
a unique (up to isomorphism) minimal model $S$. 

We shall assume from now on  that $S$ is a smooth minimal (projective) surface of general
type: this is equivalent (see \cite{arcata}) to the two conditions:

(*) $K_S^2 > 0$ and
$K_S$
   is nef 
   
   (we recall that a divisor $D$ is said to be
{\bf nef} if, for each irreducible curve $C$, we have $ D \cdot C \geq 0$).

It is very important that, as shown by Kodaira in \cite{kod-1}, the class of non  minimal surfaces is stable by small deformation;
on the other hand, a small deformation of a minimal algebraic surface of general type is again minimal (see prop. 5.5 of \cite{bpv}).
Therefore, the class of minimal algebraic surfaces of general type is stable by deformation in the large.

Even if the canonical divisor $K_S$ is nef, it does not however need to  be an ample divisor, indeed 

{\em  The canonical divisor $K_S$ of a minimal surface of general type $S$
  is ample iff
there does not exist an irreducible curve $C$ ($\neq 0$) on $S$ with $K\cdot
C=0$ $\Leftrightarrow $ there is no (-2)-curve $C$ on $S$, i.e., a curve such  that $C\cong \PP^1$, and $C^2=-2$ .}

The number of (-2)-curves  is bounded by  the rank of the Neron Severi lattice 
$NS(S)$ of $S$, and these curves can be contracted by a contraction $\pi \colon S \ra X$, where $X$ is  a normal surface which is
called the
{\bf canonical model} of $S$.

The singularities of $X$ are exactly  Rational Double Points  (in the terminology of \cite{artin}), also called  Du Val or Kleinian singularities, and $X$ is Gorenstein
with canonical divisor $K_X$ such that $ \pi^* ( K_X) = K_S$.

The canonical model
is  directly obtained from the  $5$-th pluricanonical map of $S$,
but it is abstractly defined as  the Projective Spectrum (set of homogeneous
prime ideals) of the canonical ring $$\sR (S) : =(\sR(S,K_S)) : =  \bigoplus_{m \geq 0}H^0 (\hol_S (mK_S).$$

 In fact if $S$ is a surface of general type the canonical ring
$\sR(S)$ is a graded $\CC$-algebra of finite
type (as first proven by Mumford in \cite{mum1}), and then
 the canonical model is    $X=\mbox{\rm Proj}(\sR(S,K_S))= \mbox{\rm Proj}(\sR(X,K_X))$.

By choosing a minimal homogeneous set of generators of $\sR(S)$ of degrees $d_1, \dots, d_r$ one obtains a natural embedding
of the canonical model $X$ into a weighted projective space (see\cite{dolg}). This is however not convenient in order to apply Geometric Invariant Theory, since one has then to divide by non reductive groups, unlike the case of pluricanonical maps, which we now discuss.

In this context the following is the content of the  theorem of Bombieri (\cite{bom}), which shows with a very effective estimate  the boundedness
of the family of surfaces of general type with fixed invariants $K^2_S$ and $ \chi (S) := \chi (\hol_S)$.

\begin{theo}{\bf (Bombieri)}
   Let $S$ be a minimal surface of general type,
and consider the linear system $|mK_S|$ for $m\ge 5$, or for $m=4$
when $K^2_S \ge 2$.

Then $|mK_S|$ yields a birational morphism $\varphi_m$ onto its image,
called the m-th pluricanonical map of $S$, 
which factors through the canonical model $X$ as $\varphi_m = \psi_m \circ \pi$,
and where $\psi_m$ is  the m-th pluricanonical map of $X$, associated to the   linear system $|mK_X|$, and gives an embedding of the canonical model 
$$  \psi_m \colon X \ra \cong X_m \subset \PP H^0 (\hol_X (mK_X))^{\vee} = 
 \PP H^0 (\hol_S (mK_S))^{\vee} .$$
 \end{theo}
 
 \subsection{ The Gieseker moduli space}
 
 The theory of deformations of complex spaces is conceptually simple but technically involved because
Kodaira, Spencer, Kuranishi, Grauert et al. had to prove the convergence of the 
power series solutions which they produced.

It is a matter of life that tori and algebraic K3 surfaces have small deformations which are not algebraic.
But there are  cases, like the case  of curves and of surfaces of general type, where
all small deformations are still projective, and then life simplifies incredibly,
since one can deal only with projective varieties or projective subschemes.

For these, the most natural parametrization, from the point of view of deformation theory,
is given by the Hilbert scheme, introduced by Grothendieck (\cite{groth}). Let us illustrate this concept through
the case of surfaces of general type.

For these, as we already wrote, the  first important consequence of the theorem on pluricanonical embeddings
is the finiteness, up to deformation, of the minimal surfaces $S$
of general type with fixed invariants  $\chi  (S) = a $ and $K^2_S = b$ .

In fact, their 5-canonical models $X_5$ are surfaces with
Rational Double Points as singularities and of degree $ 25 b$ in a fixed
projective space $ \PP^N$, where $ N + 1 = P_5: = h^0 (5 K_S) =
\chi (S)  + 10 K^2_S = a + 10 b$.

The Hilbert polynomial of $X_5$ equals
$$ P (m) :=  h^0 (5 m K_S) =
a + \frac{1}{2}( 5 m -1)  5 m  b .$$

Grothendieck (\cite{groth})  showed that there is

i) an integer $d$ and

ii) a subscheme $\HHH = \HHH_P$ of the
Grassmannian of codimension $P(d)$- subspaces of $ H^0 (\PP^N, \hol_{\PP^N} (d))$,
called Hilbert scheme, such that

iii) $\HHH $  parametrizes the degree $d$ graded pieces $ H^0
(\sI_{\Sigma}(d))$ of the homogeneous ideals of all the subschemes $\Sigma
\subset \PP^N$
having the given Hilbert polynomial $P$.

We can then talk about the Hilbert point of $\Sigma$ as the Pl\"ucker point 
$$ \Lambda^{P(d)} (r_{\Sigma}^{\vee})$$
$$ r_{\Sigma} :  H^0 (\PP^N, \hol_{\PP^N} (d)) \ra  H^0 (\Sigma, \hol_{\Sigma} (d))$$
being  the restriction homomorphism (surjective for $d$ large).

Inside $\HHH$ one has the open set
$$ \HHH ^0 : = \{ \Sigma | \Sigma {\rm \ is\ reduced \ with \ only \
R.D.P. 's \ as \
singularities }\}.
$$

This is plausible, since rational double points are  hypersurface singularities, and 
first of all the dimension of the Zariski tangent space is upper semicontinuous,
as well as the multiplicity: some more work is needed to show that the further property
of being a `rational' double point is open. The result has been extended in greater generality by Elkik
in \cite{elkik}.

One can use the following terminology (based on results of Tankeev in \cite{tankeev}).

\begin{defin}The 5-pseudo moduli space of surfaces of general type with
given invariants $ K^2$, $\chi$ is the closed subscheme 
$ \HHH _0 \subset \HHH^0$(defined by fitting ideals of the direct image of $\omega_{\Sigma}^{
\otimes 5} \otimes
   \hol_{\Sigma}(-1) $),
$$  \HHH _0 (\chi, K^2) : = \{ \Sigma \in \HHH^0| \omega_{\Sigma}^{
\otimes 5} \cong
   \hol_{\Sigma}(1)   \} $$

\end{defin}

Since $\HHH_0$ is a quasi-projective scheme, it has a finite number
of irreducible
components, called the
{\bf deformation
types } of the surfaces of general type with
given invariants $ K^2$, $\chi$.

As we shall see, the  above deformation types  of canonical models coincide with the equivalence classes for
the relation of deformation equivalence between minimal surfaces of general type.

\begin{rem}The group $ \PP GL (N+1 , \CC)$ acts on $\HHH_0$ with
finite stabilizers
(corresponding to the groups of automorphisms of each surface)
and the orbits correspond to the isomorphism classes of minimal
surfaces of general
type with invariants $ K^2$, $\chi$. 

Tankeev  in \cite{tankeev} showed that a quotient by this action exists
not only as a complex
analytic space, but also as a Deligne Mumford stack  (\cite{d-m}).
\end{rem}

Saying that the quotient is a stack is a way to remedy the fact that, over the locus of
surfaces with automorphisms, there does not exist a universal family, so we have only,
in Mumford's original terminology, a coarse and not a fine moduli space.

In a technically very involved paper (\cite{gieseker}) Gieseker showed   that, if one replaces
the 5-canonical embeddding by an $m$-canonical embedding with
much higher $m$, then the Hilbert point $ \Lambda^{P(d)} (r_{\Sigma}^{\vee})$ is a stable point;
this means that, beyond the already mentioned preperty that the stabilizer is finite,
that there are polynomial functions which are  invariant for the action of $ SL (N+1 , \CC)$
and which do not vanish at the point, so that the Hilbert point maps to a point of
the Projective spectrum of the ring of $ SL (N+1 , \CC)$-invariants.

The result of Gieseker leads then to the following

\begin{theo}{\bf (Gieseker)}
For m very large, the quotient $$ \mathfrak M_{\chi, K^2}^{can}: =  \HHH _0 (\chi, K^2)/ SL (N+1 , \CC)$$  exists as a quasi-projective  scheme.
It is independent of $m$ and called the {\bf Gieseker moduli space} of canonical models of surfaces of general type
with invariants $\chi, K^2$.
\end{theo}

It should be noted that at that time Gieseker only established the result for a field of characteristic zero; as he remarks in the paper,
the only thing which was missing then in characteristic $p$ was the boundedness of the surfaces of general type with given invariants $\chi, K^2$. This result was provided by Ekedahl's extension of Bombieri's theorem to characteristic $p$ (\cite{ekedahl}, see also
\cite{cf} and \cite{4auth} for a simpler proof).

\subsection{Minimal models versus canonical models}

Let us go back to the assertion  that deformation equivalence classes of minimal surfaces of general type 
are the same thing as deformation types of canonical models (a fact which is no longer true in higher dimension).

We have more precisely the following theorem.

\begin{theo}\label{can=min}

Given two minimal surfaces of general type $S, S'$ and their respective
canonical models $X, X'$, then

$S$ and $S'$ are deformation equivalent $\Leftrightarrow$ $X$ and $X'$ are
deformation equivalent.
\end{theo}

The idea of the proof can be simplified by the elementary
 observation that, in order to analyse
deformation equivalence, one may restrict oneself to the case of families parametrized by a base $T$
with $ dim (T) = 1$: since two points in a complex space $ T \subset \CC^n$ (or in an algebraic variety)
belong to the same irreducible component of $T$ if and only if they
belong to an irreducible curve $ T ' \subset T$.
And one may further reduce to the case where  $T$ is smooth simply by taking the
normalization $ T^0 \ra T_{red} \ra T$ of the reduction $T_{red}$
of $T$, and taking the pull-back of the family  to $T^0$.

But the crucial point underlying the  result is the theorem on the 
 so-called simultaneous resolution of singularities (cf. \cite{tju},\cite{brieskorn},
\cite{brieskorn2}, \cite{nice})

\begin{theo}\label{simultaneous}
{\bf (Simultaneous resolution according to Brieskorn and Tjurina).}
Let $T : = \CC^{\tau}$ be the basis of the semiuniversal deformation
of a Rational Double Point $ (X,0)$. Then there exists a ramified Galois
cover $ T' \ra T$, with $T'$ smooth $T' \cong  \CC^{\tau}$ such that the pull-back $ \X ' : = \X \times _T T'$
admits a simultaneous resolution of singularities
$ p : \SSS' \ra \X'$ (i.e., $p$ is bimeromorphic and
all the fibres of the composition $  \SSS' \ra \X' \ra T'$
are smooth and equal, for $t'_0$, to the minimal resolution
of singularities of $ (X,0)$.
\end{theo}

We reproduce Tjurina' s proof for the case of $A_n$-singularities,
observing that the case of the node was already described in the previous section.

\Proof
Assume that we have the $A_n$-singularity
$$ \{ (x,y,z) \in \CC^3 | xy = z^{n+1}  \}.$$
Then the semiuniversal deformation is given by
   $$ \X : = \{ ((x,y,z) ,(a_2, \dots a_{n+1}) ) \in \CC^3 \times \CC^n| 
xy = z^{n+1} +
a_2 z^{n-1} + \dots a_{n+1}  \} ,$$
the family corresponding to the natural deformations of the
simple cyclic covering.

We take a ramified Galois covering with group $\SSS_{n+1}$ corresponding
to the splitting polynomial of the deformed degree $n+1$ polynomial
$$\X' : =  \{ ((x,y,z), (\alpha_1, \dots \alpha_{n+1})) \in \CC^3 
\times \CC^{n+1}|
\sum_j \alpha_j = 0, \ xy = \prod_j ( z - \alpha_j)  \} .$$
One resolves the new family $\X'$ by defining
$ \phi_i : \X' \dasharrow \PP^1$ as $$ \phi_i : = (x , \prod_{j=1}^i ( z -
\alpha_j))$$
and then taking the closure of the graph of
$ \Phi : = (\phi_1, \dots \phi_n) : \X' \dasharrow (\PP^1)^n$.

\qed

Here  the Galois group $G$ of the covering $ T' \ra T$
in the above theorem is the Weyl group corresponding to the Dynkin diagram
of the singularity (whose vertices are the (-2) curves in the minimal resolution, and whose edges correspond to the intersection points). 

I.e., if  $\sG$ is the simple algebraic group
corresponding to the Dynkin diagram (see \cite{hum}), and $H$ is a
Cartan subgroup,
$N_H$ its normalizer, then the Weyl group is the factor group $W: =  N_H / H$.
For example, $A_n$ corresponds to the group $ SL(n+1, \CC)$, its Cartan subgroup
is the subgroup of diagonal matrices, which is normalized by the symmetric
group $\SSS_{n+1}$, and $N_H$ is here a semidirect product of $H$ with
$\SSS_{n+1}$.
E. Brieskorn (\cite{nice}) found later a direct
explanation of this phenomenon.

The Weyl group $W$ and  the quotient $ T = T' / W$  play a crucial role in the understanding of the relations between the deformations of
the minimal model $S$ and the canonical model $X$, which is a nice discovery by Burns and Wahl (\cite{b-w}).

But, before we do that, let us make the following important observation, saying that  the local analytic structure of the Gieseker moduli space is determined by the action of
the group of automorphisms of $X$ on the Kuranishi space of $X$.

\begin{rem}
Let $X$ be the canonical model of a minimal surface of general type $S$ with invariants 
$\chi, K^2$. The isomorphism class of $X$ defines a point  $[X] \in \mathfrak M_{\chi, K^2}^{can}$.

Then the germ of complex space $(\mathfrak M_{\chi, K^2}^{can},{[X]})$ is analytically
isomorphic to the quotient $\mathfrak B (X) / Aut (X)$ 
of the Kuranishi space of $X$ by the finite group $ Aut (X) = Aut (S)$.
\end{rem}

Forgetting for the time being about automorphisms, and concentrating on families,
we want to explain the `local contributions to global deformations of surfaces',
in the words of Burns and Wahl (\cite{b-w}).

Let $S$ be a minimal surface of general type and let $X$ be its
canonical model. To avoid confusion between the corresponding Kuranishi spaces,
denote by $\Def(S)$ the Kuranishi space for $S$, 
respectively  $\Def(X)$ the Kuranishi space of $X$.

Their result explains the relation holding between $\Def(S)$ and $\Def(X)$.

\begin{theo}{\bf  (Burns - Wahl)}
  Assume that $K_S$ is not ample and let $\pi :S \ra X$ be the
canonical morphism.

   Denote by $\mathcal{L}_X$ the space of local deformations of the
singularities of $X$ (Cartesian product of the corresponding Kuranishi  spaces) and by
$\mathcal{L}_S$ the space of deformations of a neighbourhood of the
exceptional locus  of $\pi$. Then
$\Def(S)$ is realized as the fibre product associated to the Cartesian diagram

\begin{equation*}
\xymatrix{
\Def(S) \ar[d]\ar[r] & Def (S_{Exc(\pi)})= :  \mathcal{L}_S \cong \CC^{\nu}, \ar[d]^{\lambda} \\
L \colon \Def(X) \ar[r] & Def (X_{Sing X}) = : \mathcal{L}_X \cong \CC^{\nu} ,}
\end{equation*}
where $\nu$ is the number of rational $(-2)$-curves in $S$, and
$\lambda$ is a Galois covering
with Galois group $W := \oplus_{i=1}^r W_i$, the direct sum of the
Weyl groups $W_i$ of the singular points of $X$ (these are generated by reflections, hence yield
a smooth quotient, see \cite{chevalley}).
\end{theo}

An immediate consequence is the following

\begin{cor}{\bf  (Burns - Wahl)}
  1) $\psi:\Def(S) \ra \Def(X)$ is a finite morphism, in
particular, $\psi$ is surjective.

\noindent
2) If the derivative of $\Def(X) \ra \mathcal{L}_X$ is not surjective (i.e., the
singularities of $X$ cannot be independently  
   smoothened  by the first order infinitesimal deformations of $X$),
   then $\Def(S)$ is singular.
\end{cor}

Moreover one has a further corollary

\begin{cor}{\bf  \cite{cat5}}\label{ENR}

If the morphism $L$ is constant, 
   then $\Def(S)$ is everywhere non reduced, 
   $$ \Def(S) \cong \Def(X) \times \lambda^{-1} (0).$$
\end{cor}

In  \cite{cat5} several examples were exhibited, extending two previous examples by Horikawa and Miranda.
In these examples the canonical model $X$ is a hypersurface of degree $d$ in a weighted projective space:

$$ X_d \subset \PP (1,1,p,q) , d > 2 + p + q, $$
where
\begin{itemize}
\item
$ X_d \subset \PP (1,1,2,3) $, $ d = 1 + 6 k$, $X$ has one singularity of type $A_1$ and one of type $A_2$, or
\item
$ X_d \subset \PP (1,1,p,p+1) $, $ d = p (k (p+1) -1)$, $X$ has one singularity of type $A_p$, or
\item
$ X_d \subset \PP (1,1,p,rp-1) $, $ d =(kp-1) (rp-1)$, $ r > p-2$, $X$ has one singularity of type $A_{p-1}$.
\end{itemize}

The philosophy in these examples (some hard calculations are however needed) is that all the deformations of $X$ remain hypersurfaces in the same projective space,
and this forces $X$ to preserve, in view of the special arithmetic properties of the weights and of the degree, its singularities.

\subsection{Number of moduli done right}

The interesting part of the discovery of Burns and Wahl is that they completely clarified the background of an old dispute going on
in the late 1940's between Francesco Severi and Beniamino Segre. The (still open) question was: given a degree $d$,
which is the maximum number $\mu (d)$ of nodes that a normal surface $X \subset \PP^3$ of degree $d$
can have ?

The answer is known only for small degree $ d\leq 6$: $\mu (2)=1$, $\mu (3)=4$ (Cayley's cubic), $\mu (4)=16$ (Kummer surfaces), $\mu (5)=31$ (Togliatti quintics), $\mu (6)=65$ (Barth's sextic), and   
Severi made the following bold assertion: an upper bound is clearly given by the `number of moduli',
i.e., the dimension of the moduli space
of the surfaces of degree $d$ in $\PP^3$; this number equals the difference between the dimension of the underlying projective space $\frac{(d+3)(d+2)(d+1)}{6} -1$
and the dimension of the group of projectivities, at least for $ d \geq 4$ when the general surface
of degree $d$ has only a finite group of projective automorphisms.

One should then have $ \mu (d) \leq  \nu (d) : = \frac{(d+3)(d+2)(d+1)}{6} -16$, but Segre (\cite{segre}) found some easy examples
contradicting this inequality, the easiest of which are  some surfaces of the form 
$$ L_1(x)\cdot   \dots \cdot L_{d}(x) - M (x)^2,$$
where $d$ is even, the $L_i(x)$ are linear forms, and $M(x)$ is a homogeneous polynomial of degree $\frac{d}{2}$.

Whence the easiest Segre surfaces have $ \frac{1}{4} d^2 (d-1)$ nodes, corresponding to the points where
$ L_i (x) = L_j (x) = M(x)= 0$, and this number grows asymptotically as $ \frac{1}{4} d^3$, versus Severi's upper bound,
which grows like $ \frac{1}{6} d^3$ (in fact we know nowadays, by Chmutov in \cite{chmutov}, resp. Miyaoka in \cite{miyaoka}, that  $  \frac{5}{12} d^3 \leq \mu(d) \leq  \frac{4}{9} d^3$).

The problem with Severi's claim is simply that the nodes impose independent conditions infinitesimally,
but only for the smooth model $S$: in other words, if $X$ has $\de$ nodes, and $S$ is its desingularization,
then $Def(S)$ has Zariski tangent dimension at least $\de$, while it is not true that  $Def(S)$ has dimension
at least $\de$. Burns and Wahl, while philosophically rescueing Severi' s intuition, showed in this way that
there are a lot of examples of obstructed surfaces $S$, thereby killing Kodaira and Spencer's dream that
their cohomology dimension $h^1 (\Theta_S)$ would be the expected number of moduli.

\subsection{The  moduli space for minimal models of surfaces of general type}

In this section we shall derive some further moral consequences from the result of Burns and Wahl.

For simplicity, consider the case where the canonical model $X$ has only one double point, and recall the notation
introduced previously, concerning the local deformation of the node, given by the family 
$$  uv = w^2 - t,$$
the pull back family 
$$\XX = \{ (u,v,w,\tau ) |  w^2
- \tau^2 = uv \}$$

and the two families 
$$\sS : \{ (u,v,w,\tau)(\xi) \in  \XX \times \PP^1| \
\frac{w-\tau}{u} =   \frac{v}{w+\tau} = \xi \}$$
$$\sS' : \{ (u,v,w,\tau)(\eta) \in  \XX \times \PP^1| \
\frac{w+\tau}{u} =   \frac{v}{w-\tau} = \eta \}.$$

There are two cases to be considered in the following oversimplified example:

1)  $t  \in \Delta$ is a coordinate of an effective smoothing of the node, hence we have a family $\sS$ parametrized by $ \tau \in \De$

2) we have no first order smoothing of the node, hence    the Spectrum of the ring $\CC[\tau]/ (\tau^2)$ replaces $\Delta$.

In case 1), we have two families $\sS, \sS'$ on a disk $ \Delta$ with coordinate $\tau$, which are isomorphic over the punctured
disk. This induces for the punctured disk $ \Delta$ with coordinate $\tau$, base of the first family, an equivalence relation induced by the isomorphism with itself   where $\tau$ is exchanged with $- \tau$; in this case the ring of germs of holomorphic functions at the origin which 
are invariant for the resulting equivalence relation is just the ring of power series $\CC\{t\}$, and we have a smooth `moduli space'.

In case 2), there is no room for identifying $\tau$ with $- \tau$, since if we do this on  $\CC[\tau]/ (\tau^2)$, which has only one point, 
we glue the families without inducing the identity on the base, and this is not allowed. In this latter case we are left with the non reduced scheme
$Spec (\CC[\tau]/ (\tau^2))$ as a `moduli space'.

Recall now Mumford's definition of a coarse moduli space (\cite{mumford}, page 99, definition 5.6 , and page 129, definition 7.4) for a functor
of type $ \sS urf$, such as $ \sS urf^{min} $, associating to a scheme $T$ the  set $ \sS urf^{min} (T)$ of isomorphism classes of families of smooth minimal surfaces of general type over $T$,  or as $ \sS urf^{can} $, associating to a scheme $T$ the  set $ \sS urf^{can} (T)$ of isomorphism classes of families of canonical models of surfaces of general type over $T$.
 
It should be a scheme $A$, given together  with a morphism $\Phi$ from the functor $ \sS urf $ to $h_A: = \sH om (-, A)$,
such that 

\begin{enumerate}
\item
for all algebraically closed fields $k$, $\Phi(Spec \ k) \colon  \sS urf (Spec \  k) \ra h_A (Spec \ k)$ is an isomorphism

\item

any other morphism $\Psi$ from the functor $ \sS urf $ to a functor $h_B$ factors uniquely through
$ \chi \colon h_A \ra h_B$.

\end{enumerate}

Since any family of canonical models $p \colon \sX \ra T$ induces, once we restrict $T$ and we choose a local frame for  the direct image sheaf $p_* (\omega_{ \sX | T}^m)$ a family of pluricanonical models embedded in a fixed $\PP^{P_m-1}$, follows

\begin{theo}\label{Gieseker=coarse}
The Gieseker moduli space $ \mathfrak M_{\chi, K^2}^{can}$ is the coarse moduli space   for the functor  $ \sS urf^{can}_{\chi, K^2} $, i.e., for
canonical models of surfaces $S$ of general type with given invariants  $ \chi, K^2$.
Hence it  gives a natural complex structure on the topological space  $ \mathfrak M(S)$,  for $S$ as above.
\end{theo} 

As for the case of algebraic curves, we do not have a fine moduli space, i.e., the functor is not representable
by this scheme. Here, automorphisms are the main obstruction to the existence
of a fine moduli space: dividing the universal family over the Hilbert scheme by the linear group we obtain a family
over the quotient coarse moduli space such that the fibre over the isomorphism class of a canonical model $X$, in the case where
the  group of automorphisms $Aut (X)$ is non trivial,
 is the quotient 
$ X / Aut (X)$.  And $ X / Aut (X)$ is then not isomorphic to $X$.

Instead, in  the case of the functor  $ \sS urf^{min} (T)$, there is a further  problem: that the equivalence relation (of isomorphism of families) is not proper on the parameter space, as we already mentioned.

While for curves we have a Deligne-Mumford stack, which amounts roughly speaking to take more general functors than functors which are  set valued,  this no longer holds for surfaces of general type.
Therefore Artin in \cite{artinstacks} had to change the definition, allowing more general equivalence relations.
The result is (\cite{artinstacks}, Example 5.5 page 182)

\begin{theo}{\bf (Artin)}
There exists a moduli space $ \mathfrak M_{\chi, K^2}^{min}$ which is an algebraic Artin stack  for minimal surfaces of general type with given invariants  $ \chi, K^2$.
\end{theo} 

The beginning step for Artin was   to show that there is a finite number of algebraic families parametrizing all the minimal models with given invariants: this step is achieved by Artin in \cite{ArtinBrieskorn} showing that the simultaneous resolution of a family of canonical models can be done not only in the holomorphic category, but also over schemes, provided  that one allows a base change 
process   producing locally quasi-separated algebraic spaces. 

After that,  one can consider the equivalence relation given by isomorphisms of families.

We shall  illustrate the local picture  by  considering  the restriction of this equivalence relation to the base $Def(S)$ of the Kuranishi
family.

(I) First of all we have an action of the group $Aut(S)$ on $Def(S)$, and we must take the quotient by this action.

In fact, if $g \in  Aut(S)$, then $g$ acts on $Def(S)$, and if $\sS \ra Def(S)$ is the universal family, we can take the pull back family $g^* \sS $.
By the universality of the Kuranishi family, we have an isomorphism $g^* \sS \cong \sS $ lying over the identity of  $Def(S)$,
and by property (2) we must  take the quotient  of  $Def(S)$ by this action of $Aut(S)$.

(II) Let now $w \in W$ be an element of the Weyl group which acts on $Def(S)$ via the Burns-Wahl fibre product. We let $U_w$
be the open set of $\sL_S$ where the transformation $w$ acts freely (equivalently, $w$ being a pseudoreflection, $U_w$
is the complement of the hyperplane of fixed points of $w$), and  we let $Def(S)_w$ be equal to the open set inverse image of $U_w$.

Since the action of $w$ is free on $Def(S)_w$, we obtain that $w$ induces an isomorphism of the family $\sS_w \ra Def(S)_w$
with its pull back under $w$, inducing the identity on the base: hence we have to take the graph of $w$ on $Def(S)_w$,
and divide $Def(S)_w$ by the action of $w$.

(III) The 'equivalence relation' on $Def(S)$ is thus generated by (I) and (II), but it is not really a proper equivalence relation.

The  complex space underlying  $ \mathfrak M_{\chi, K^2}^{min}$ is obtained taking the subsheaf  of $\hol_{Def(S)}$ consisting of the  functions which are
invariant for this equivalence relation (i.e., in case (II) , their restriction to  $Def(S)_w$ should be $w$- invariant).

$ \mathfrak M_{\chi, K^2}^{min}$  has the same associated reduced complex space as  
$ \mathfrak M_{\chi, K^2}^{can}$,
but a different ringed space structure, as the examples of \cite{cat5} mentioned after corollary \ref{ENR} show, see the next subsection.

In fact, the main difference is that $ \mathfrak M_{\chi, K^2}^{can}$ is locally, by Burns-Wahl's fibre product theorem, the quotient of
$Def(S)$ by the group $ G'$ which is the semidirect product of the Weyl group $W$ by $Aut(S) = Aut(X)$,
as a  ringed space (the group $G'$ will make its appearance again in the concrete situation of Lemma \ref{nolift}).

Whereas, for $ \mathfrak M_{\chi, K^2}^{min}$ the action on the set is the same, but the action of an element $w$ 
of the Weyl group on the sheaf of regular functions 
 is only there on  an open set ( $Def(S)_w$)
and this set  $Def(S)_w$ may be  empty if  $Def(X)$ maps to a branch divisor of the quotient map $\sL_S \ra \sL_X$. 

A general question  for which we have not yet found the time to provide an answer is whether there is a quasi-projective scheme whose
underlying complex space is  $ \mathfrak M_{\chi, K^2}^{min}$: we suspect that the answer should be positive.

\subsection{Singularities of moduli spaces}

In general one can define

\begin{defin}
The local moduli space  $ (\mathfrak M_{\chi, K^2}^{min, loc},{[S]})$ of a smooth minimal surface of general type $S$ 
is the quotient  $ Def(S) / Aut(S)$ of the Kuranishi space of $S$ by the action of the finite group of automorphisms of $S$.

\end{defin}

Caveat: whereas,  for the canonical model $X$, $ Def(X) / Aut(X)$ is just the analytic germ of the Gieseker moduli space at the point corresponding to the isomorphism class of $X$, the local moduli space $ (\mathfrak M_{\chi, K^2}^{min, loc},{[S]}) : =  Def(S) / Aut(S)$ 
 is different in general from the analytic germ of the moduli space  $( \mathfrak M_{\chi, K^2}^{min},{[S]})$, though it surjects onto the latter. 
 But it is certainly equal to it 
in the special case where the surface $S$ has ample canonical divisor $K_S$.

The Cartesian diagram by Burns and Wahl  was used in \cite{cat5} to construct everywhere non reduced moduli spaces
$ (\mathfrak M_{\chi, K^2}^{min, loc},{[S]}) : =  Def(S) / Aut(S)$  for minimal models of surfaces of general type.

In this case the basic theorem is 

\begin{theo}{\bf (\cite{cat5})}
There are (generically smooth) connected components of  Gieseker moduli spaces $ \mathfrak M_{\chi, K^2}^{can}$ such that all the canonical models in it are singular. 

Hence the local moduli spaces $ (\mathfrak M_{\chi, K^2}^{min, loc},{[S]}) $ for the corresponding minimal models
 are everywhere non reduced, and the same occurs for the germs  $( \mathfrak M_{\chi, K^2}^{min},{[S]})$.

\end{theo}

The reason is simple: we already mentioned that  if we take the fibre product 

\begin{equation*}
\xymatrix{
\Def(S) \ar[d]\ar[r] & Def (S_{Exc(\pi)})= :  \mathcal{L}_S \cong \CC^{\nu}, \ar[d]^{\lambda} \\
\Def(X) \ar[r] & Def (X_{Sing X}) = : \mathcal{L}_X \cong \CC^{\nu} ,}
\end{equation*}

the lower horizontal arrow maps to a reduced point, hence $Def(S)$ is just the product
$Def(X) \times \lambda^{-1} (0)$, and 
$\lambda^{-1} (0)$ is a non reduced point, spectrum of an Artin local ring of length equal to 
the cardinality of the
Galois group $W := \oplus_{i=1}^r W_i$.

Hence $\Def(S)$ is everywhere non reduced. Moreover, one can show that, in the examples considered, the general surface
has no automorphisms, i.e., there is an open set for which the analytic germ of the Gieseker moduli space  coincides 
with the Kuranishi family $\Def(X)$, and the family of canonical models just obtained is equisingular.

Hence, once we consider the equivalence relation on $\Def(S)$ induced by isomorphisms, the Weyl group acts trivially (because
of equisingularity, we get  $\Def(S)_w = \emptyset$ $\forall w \in W$). Moreover, by our choice of a general open set, $Aut(S)$ is trivial.

The conclusion is that   $( \mathfrak M_{\chi, K^2}^{min},{[S]})$ is locally isomorphic to  the Kuranishi family $\Def(S)$,
hence everywhere non reduced.
\qed

Using an interesting result of M'nev about  line configurations, Vakil (\cite{murphy}) was able to show that `any type of singularity'
can occur for the Gieseker moduli space (his examples are such that $S = X$ and $S$ has no automorphisms, hence they
produce the desired singularities also for the local moduli space,  for the Kuranishi families; they also produce singularities  for the Hilbert schemes,
because his surfaces have $q(S) : = h^1(\hol_S) = 0$).

\begin{theo}{\bf (Vakil's `Murphy's law')}
Given any singularity germ of finite type over the integers, there is a  Gieseker moduli space $ \mathfrak M_{\chi, K^2}^{can}$ and a surface $S$ with ample canonical divisor $K_S$ (hence $S = X$) such that $ ( \mathfrak M_{\chi, K^2}^{can}, [X])$ realizes the given singularity germ.

\end{theo}

In the next section we shall see  more instances where automorphisms play an  important role.

\section{Automorphisms and moduli}

\subsection{Automorphisms and canonical models}

The good thing about differential forms is that any group action on a complex manifold leads to
a group action on the vector spaces of differential forms.

Assume that $G$ is a group acting on a surface $S$ of general type, or more generally on a K\"ahler manifold $Y$: then 
$G$ acts linearly
on the Hodge vector spaces $H^{p,q}(Y)  \cong H^q (\Omega^p_Y)$, and also on the vector spaces
$ H^0 ( (\Omega^n_Y)^{\otimes m}) = H^0 (\hol_Y (mK_Y))$, hence on the canonical ring 
 $$\sR (Y) : =(\sR(Y,K_Y)) : =  \bigoplus_{m \geq 0}H^0 (\hol_Y (mK_Y)).$$
 
 If $Y$ is a variety of general type, then  $G$ acts linearly on the vector space $H^0 (\hol_Y (mK_Y))$,
 hence linearly on the m-th pluricanonical image $Y_m$, which is an algebraic variety bimeromorphic to $Y$.
 Hence $G$ is contained in the algebraic group $ Aut (Y_m)$ and, if $G$ were infinite, as observed by 
 Matsumura (\cite{matsumura}), $ Aut (Y_m)$ would contain a non trivial Cartan subgroup (hence $\CC$ or $\CC^*$) and
 $Y$ would be uniruled, a contradiction. This was the main argument of the following
 
 \begin{theo}{\bf (Matsumura)}
 The automorphism group of a variety $Y$ of general type is finite.
 
 \end{theo}
 
Let us specialize now to $S$ a surface of general type, even if most of what we say shall remain valid also in higher
dimension.

Take an $m$-pseudo moduli space $  \HHH _0 (\chi, K^2)$ with $m$ so large that the corresponding Hilbert points
of the varieties $X_m$ are stable, and let $G$ be a finite group acting on a minimal surface of general type $S$
whose $m$-th canonical image is in $  \HHH _0 (\chi, K^2)$.

Since $G$ acts on the vector space $V_m : = H^0 (\hol_S (mK_S))$, the vector space splits uniquely, up to permutation of the summands,   as a direct sum
of irreducible representations
$$ (**) \ \ V_m = \bigoplus_{\rho \in Irr (G)} W_{\rho}^{n(\rho)}.$$

We come now to the basic notion of a family of $G$-automorphisms

\begin{defin}
A family of $G$-automorphisms is a triple $$ ((p \colon \sS \ra T), G, \alpha )$$ where:
\begin{enumerate}
\item
$ (p \colon \sS \ra T)$ is a family in a given category (a smooth family for the case of minimal models of general type)
\item
$G$ is a (finite) group
\item
$\alpha \colon G \times \sS \ra \sS $ yields a biregular  action  $ G \ra Aut (\sS)$, which is compatible with the projection $p$ and with the trivial action of $G$ on the base $T$ (i.e., $p(\alpha (g,x)) = p(x) , \ \forall g \in G, x \in \sS$).

\end{enumerate}

As a shorthand notation, one may also write $ g(x)$ instead of $\alpha (g,x)$, and by abuse of notation say that the family
of automorphisms is a deformation of the pair $(S_t,G)$ instead of the triple  $(S_t,G, \alpha_t)$.

\end{defin}

\begin{prop}\label{actiontype}

1) A family of automorphisms of  surfaces of general type (not necessarily minimal models) induces
a family of automorphisms of canonical models.

2) A family of automorphisms of canonical models induces, if the basis $T$ is connected, 
a constant decomposition type $(**)$ for $V_m (t)$.

3) A family of automorphisms of  surfaces of general type admits a differentiable trivialization,
i.e., in a neighbourhood of $ t_0 \in T$,  a diffeomorphism as a family with $ ( S_0 \times T, p_T, \alpha_0 \times Id_T)$;
in other words, with the trivial family for which $ g (y,t) = (g(y),t)$.

\end{prop}

\Proof

We sketch only the main ideas.

1) follows since one can take the relative canonical divisor  $ K : = K_{\sS|T} $,
the sheaf of graded algebras $$ \sR (p): = \oplus_m p_* (\hol_{\sS} (mK))$$ and 
 take the relative Proj, yielding $ \XX : = Proj (\sR (p) )$, whose fibres are the canonical models.
 
 2) follows since for a representation space $(V, \rho')$ the multiplicity with which
 an irreducible representation $W$ occurs in $V$ is the dimension of $ Hom(W,V)^G$,
 which in turn is calculated by the integral on $G$ of the trace of $\rho''(g)$, where $\rho''(g)$
 is the representation $ Hom(W,V)$.
 If we have a family, we get a continuous integer valued function, hence a constant function.
 
 3) Since $G$ acts trivially on the base $T$, it follows that for each $g \in G$ the fixed locus
 $ Fix(g)$ is a relative submanifold with a submersion onto $T$. By the use of stratified spaces (see \cite{mather}), and control data,
 one finds then a differentiable  trivialization for the quotient analytic space $ \sS / G$, hence a trivialization of the action.

\qed

Let us then consider the case of a family of canonical models: by 2) above, and shrinking the base in order to
make the addendum $ \sR (p)_m  = p_* (\hol_{\sS} (mK))$ free, we get an embedding of the family
$$ (\XX, G)  \hookrightarrow T \times (\PP ( V_m = \bigoplus_{\rho \in Irr (G)} W_{\rho}^{n(\rho)} ), G).$$
In other words, all the canonical models $X_t$ are contained in a fixed projective space, where also
the action of $G$ is fixed.

Now, the canonical model $X_t$ is left invariant by the action of $G$ if and only if its Hilbert point is
fixed by $G$. Hence, we get a closed set 

$  \HHH _0 (\chi, K^2)^G \subset \HHH _0 (\chi, K^2)$

of the pseudomoduli space, and a corresponding closed subset of the moduli space. Hence we get the following theorem.

\begin{theo}\label{can-aut=closed}
The surfaces of general type which admit an action of a given pluricanonical type $(**)$
i.e., with a fixed irreducible G- decomposition of their canonical ring, form a closed 
subvariety $( \mathfrak M_{\chi, K^2}^{can})^{G, (**)}$ of the moduli space $ \mathfrak M_{\chi, K^2}^{can}$.
\end{theo}

We shall see that the situation for the minimal models is different, because then the subset of the moduli space
where one has a fixed differentiable type is not closed.

\subsection{ Kuranishi subspaces for automorphisms of a fixed type}

Proposition \ref{actiontype} is quite useful when one analyses the deformations of a given $G$-action.

In the case of the canonical models, we just have to look at the fixed points for the action 
on a subscheme of the Hilbert scheme; whereas, for the case of the deformations of the minimal model,
we have to look at the complex structures for which the given differentiable action is biholomorphic.
Hence we derive

\begin{prop}
Consider a fixed action of a finite group $G$ on a minimal surface of general type $S$,
and let $X$ be its canonical model.
Then we obtain closed subsets of the respective Kuranishi spaces,
corresponding to deformations which preserve the given action, 
and yielding a maximal family of deformations of the $G$-action.

These subspaces are $ \mathfrak B (S) \cap H^1(\Theta_S)^G = Def(S)  \cap H^1(\Theta_S)^G$,
respectively $ \mathfrak B (X) \cap {\rm Ext}^1(\Omega^1_X, \hol_X)^G =  Def (X) \cap {\rm Ext}^1(\Omega^1_X, \hol_X)^G$.
\end{prop}

We refer to \cite{montecatini} for a proof of the first fact, while for the second 
the proof is based again on Cartan's lemma (\cite{cartan}), that the action of a finite
group in an analytic  neighbourhood of a fixed point can be linearized.

Just a comment about the contents of the proposition: it says that in each of the two cases,
the locus where a group action of a fixed type is preserved is a locally closed set of the
moduli space. We shall see more clearly the distinction in the next subsection.

\subsection{Deformations of automorphisms differ for canonical and for minimal models}

The scope of this subsection is to illustrate the main principles of a rather puzzling phenomenon
which we discovered in my joint work with Ingrid Bauer (\cite{burniat2}, \cite{burniat3})
on the moduli spaces of Burniat surfaces.

Before dwelling on the geometry of these surfaces, I want to explain clearly what happens, 
and it suffices to take the example of  nodal secondary Burniat surfaces, which I will denote
by BUNS in order to abbreviate the name. 

For BUNS one has $K^2_S = 4, p_g (S) : =  h^0(K_S) = 0$, and the bicanonical map is
a Galois cover of the Del Pezzo surface $Y$ of degree 4 with just one node as singularity
(the resolution of $Y$  is the blow up $Y'$ of the plane in 5 points, of which exactly 3 are collinear).
The Galois group is $G = (\ZZ/ 2)^2$, and over
the node of $Y$ lies a node of the canonical model $X$ of $S$, which does not have other singularities.

Then we have BUES, which  means extended secondary Burniat surfaces, whose  bicanonical map is again
a finite $(\ZZ/ 2)^2$ - Galois cover of the 1-nodal Del Pezzo surface $Y$ of degree 4 (and for these $S=X$,
i.e., the canonical divisor $K_S$ is ample).

All these actions on the canonical models fit together into a single family, but, if we pass to the minimal models, then
the topological type of the action changes in a discontinuous way when we pass from the closed set of BUNS
to the open set of BUES, and we have precisely two families.

We have , more precisely, the following theorems (\cite{burniat3}):

\begin{theo}\label{main1}{\bf (Bauer-Catanese)}
 An irreducible connected component, normal, unirational
          of  dimension 3 of the  moduli
space of 
surfaces of general type $\mathfrak M^{can}_{1,4}$ is given by the subset 
     $\sN \sE \sB_4$, formed by the disjoint  
union of the open set corresponding to BUES
({\em extended} secondary Burniat surfaces), with the
irreducible closed set parametrizing BUNS (nodal secondary Burniat surfaces).

For all surfaces $S$ in $\sN \sE \sB_4$ the bicanonical map of  the canonical model $X$ is a finite
cover of degree 4, with Galois group  $G = (\ZZ/ 2)^2$,
of the 1-nodal Del Pezzo surface $Y$ of degree 4 in $\PP^4$.

Moreover  the Kuranishi space $\mathfrak B (S)$  of any
such a minimal model $S$ is smooth.

\end{theo}

\begin{theo}\label{path}{\bf (Bauer-Catanese)}
   The deformations of nodal secondary Burniat surfaces (secondary means that  $K^2_S =4$) to
extended  secondary Burniat surfaces 
yield examples where $\Def(S,(\ZZ/2\ZZ)^2) \ra \Def(X,(\ZZ/2\ZZ)^2)$
is not surjective.

Indeed the pairs $(X,G)$, where $G : = (\ZZ/2\ZZ)^2$ and $X$ is the canonical model
of an  extended or nodal secondary Burniat surface, where the action of $G$ on $X$ 
is induced by the bicanonical map of $X$, belong to only one 
deformation type.

If $S$ is a BUNS, then  $\Def(S,(\ZZ/2\ZZ)^2) \subsetneq \Def(S)$, and 
$\Def(S,(\ZZ/2\ZZ)^2)$ consists exactly of all the BUNS ';
while for the canonical model $X$ of $S$ we have: $\Def(X,(\ZZ/2\ZZ)^2) = \Def(X)$.

Indeed for  the pairs $(S,G)$, where $S$ is the 
minimal model
of an  extended or nodal Burniat surface,  $G : = (\ZZ/2\ZZ)^2$ and  the action is induced by
the bicanonical map  (it is unique up to automorphisms of $G$), they  belong to exactly two 
distinct deformation
types, one given by BUNS, and the other given by BUES.
  
  \end{theo}
  
The discovery of BUES came later as a byproduct of the investigation of tertiary (3-nodal) Burniat surfaces,
where we knew by the Enriques-Kuranishi inequality that tertiary Burniat surfaces cannot form a component
of the moduli space: and knowing that there are other deformations helped us to find them eventually.

 For BUNS, we first erroneously thought (see \cite{burniat2}) that they form a
connected component of the moduli space, because $ G= (\ZZ/2\ZZ)^2 \subset Aut(S)= Aut(X)$ for a BUNS,
and BUNS are exactly the surfaces $S$ for which the action deforms, while we proved that for
all deformations of the canonical model $X$ the action deforms.

The description of BUNS and especially of  BUES is complicated, so I refer simply to \cite{burniat3}; but 
the essence of the pathological behaviour can be understood from the local picture around the node of
 the Del Pezzo surface  $Y$.
 
 We already described most of this local picture in a previous section. 
 
 We make here a first  additional observation:
 
 \begin{prop}\label{def-act}
Let $t \in \CC$ , and consider the action of $G: = (\ZZ/2 \ZZ)^2$ on $\CC^3$
generated by $\sigma_1(u,v,w) = (u,v,-w)$, $\sigma_2(u,v,w) = (-u,-v,w)$.
Then the hypersurfaces $X_t = \{ (u,v,w)| w^2 = uv + t\}$ are $G$-invariant,
and the quotient $X_t / G $ is the hypersurface
$$ Y_t = Y_0 = Y :=   \{ (x,y,z)| z^2 = xy\} ,$$
which has a nodal singularity at the point $x=y=z=0$.

$X_t \ra Y$ is a family of finite bidouble coverings (Galois coverings with group  $G: = (\ZZ/2 \ZZ)^2$). 

We get in this way a flat family of (non flat) bidouble covers.

\end{prop}

\begin{proof}
The invariants for the action of  $G$ on $\CC^3 \times \CC$ are:
$$ x: =u^2, y:  = v^2, z : = uv , s: = w^2, t.$$

Hence the family $\XX$  of the hypersurfaces $X_t$ is the inverse image of the
family of hypersurfaces $ s = z +t$ on the product
$$Y  \times \CC^2 = \{(x,y,z,s,t)| xy = z^2 \} .$$
Hence the quotient of $X_t$ is isomorphic to $Y$.

\end{proof}

The following is instead a rephrasing and a generalization of the discovery of
Atiyah  in the context of automorphisms, which is the main local content of the above theorem.
It says that the family of automorphisms of the canonical models $X_t$, i.e., the automorphism  group 
of the family $\XX$,   does
not lift, even after base change, to the family $\sS$ of minimal surfaces $S_{\tau}$.

\begin{lemma}\label{nolift}

Let $G$ be the group $G \cong (\ZZ/2 \ZZ)^2$ acting on $\XX$ trivially on
the variable
$\tau$, and else as follows on $\XX$:  the action of $G: = (\ZZ/2 \ZZ)^2$ on $\CC^3$
is generated by $\sigma_1(u,v,w) = (u,v,-w)$, $\sigma_2(u,v,w) = (-u,-v,w)$
(we set then $\sigma_3 := \sigma_1 \sigma_2$, so that
$\sigma_3(u,v,w) = (-u,-v,-w)$).

The invariants for the action of  $G$ on $\CC^3 \times \CC$ are:
$$ x: =u^2, y:  = v^2, z : = uv , s: = w^2, t.$$

Observe that the hypersurfaces $X_t = \{ (u,v,w)| w^2 = uv + t\}$ are $G$-invariant,
and the quotient $X_t / G $ is the hypersurface
$$ Y_t \cong Y_0 =  \{ (x,y,z)| z^2 = xy\} ,$$
which has a nodal singularity at the point $x=y=z=0$.

Let further $\sigma_4$ act by $\sigma_4(u,v,w,\tau) = (u,v,w,- \tau) $,
        let $G' \cong (\ZZ/2 \ZZ)^3$ be the group generated by $G$ and 
$\sigma_4$,
        and let $H \cong (\ZZ/2 \ZZ)^2$ be the subgroup $\{ Id, \sigma_2,
\sigma_1\sigma_4 , \sigma_3\sigma_4\}$.

The biregular action of $G'$ on $\XX$ lifts only to a birational
action on $\sS$,
respectively $\sS'$. The subgroup $H$ acts on $\sS$, respectively $\sS'$,
as a group of biregular automorphisms.

The elements of
$ G' \setminus H = \{ \sigma_1, \sigma_3, \sigma_4 , \sigma_2\sigma_4\}$
yield isomorphisms between $\sS$ and  $\sS'$.

The group $G$ acts on the punctured family $\sS \setminus S_0$,
in particular it acts on each fibre $S_{\tau}$.

Since $\sigma_4$ acts trivially on $S_0$,
the group $G'$ acts on $S_0$ through its direct summand $G$.

The biregular actions of $G$ on $\sS \setminus S_0$ and on $ S_0$
do not patch together to a biregular action on $\sS$, in particular
$\sigma_1$ and $\sigma_3$ yield birational maps which are
not biregular: they are called Atiyah flops (cf. \cite{atiyah}).
\end{lemma}

Another more geometrical way to see that there is no $G$-action on
the family $\sS$ is the following: if $G$ would act on $\sS$, and trivially on the base,
then the fixed loci for the group elements would be submanifolds with a smooth map onto 
the parameter space $\CC$ with parameter $\tau$. Hence all the quotients $ S_{\tau} / G$
would be homeomorphic.

But for BUNS the  quotient of $S_0$ by $G$ is the blow up $Y'$ of $Y$ at the origin,
while for $\tau \neq 0$, $ S_{\tau} / G$ is just $Y$! \footnote{In the case of BUNS, $Y$ is a nodal Del Pezzo surface of degree $4$,
whereas in the local analysis we use the same notation $Y$ for the quadric cone, which is the germ of the nodal Del Pezzo surface at 
the nodal singular point. }
In fact, if one wants to construct the family
of smooth models as a family of bidouble covers of a  smooth surface, one has to take
the blown up surface $Y'$ and its exceptional divisor $N$ ($N$ is called the nodal curve).

\begin{rem}
i) The simplest way to view $X_t$ is to  see $\CC^2$ as a double cover
of $Y$ branched only at the origin,
and then $X_t$ as a family of double covers of $\CC^2$
branched on the curve $ uv + t = 0$, which acquires a double point for $t=0$.

ii) If we pull back the bidouble cover $X_t$ to $Y'$, and we
normalize it,
we can see that the three branch divisors, corresponding to the fixed points for the three
non trivial elements of the group $G$, are as follows:
\begin{itemize}
\item
$D_3$ is, for $t=0$, the nodal curve $N$, and is the empty divisor
for $t\neq 0$;
\item
$D_1$ is, for $t \neq 0$, the inverse image of the curve $z + t = 0$;
while, for
$t=0$, it is only its strict transform, i.e. a divisor  made up of
$F_1, F_2$,  the proper transforms of the two branch lines 
(\{x=z=0\}, resp. \{y=z=0\}) on the quadric cone $Y$
\item
$D_2$ is an empty divisor for $t=0$, and the nodal curve $N$
for $t\neq 0$.

\end{itemize}

\end{rem}

The above remark shows then that  in order to construct the smooth models, one has first of all
to take a discontinuous family of branch divisors; and, moreover, for $ t \neq 0$,
we obtain then a non minimal surface which contains two (-1)-curves ($S_t = X_t$
is then gotten by contracting these two (-1)-curves).

\medskip

\subsection{Teichm\"uller space for surfaces of general type}

Recall the fibre product considered by Burns and Wahl:

\begin{equation*}
\xymatrix{
\Def(S) \ar[d]\ar[r] & Def (S_{Exc(\pi)})= :  \mathcal{L}_S \cong \CC^{\nu}, \ar[d]^{\lambda} \\
\Def(X) \ar[r] & Def (X_{Sing X}) = : \mathcal{L}_X \cong \CC^{\nu} ,}
\end{equation*}

This gives a map $f \colon  Def (S) \ra Def (X) / Aut (X)$ of the Kuranishi space of $S$ into an open set of a quasiprojective variety,
which factors through Teichm\"uller space.

\begin{theo}\label{kur=teich-surf}
Let $S$ be the minimal model of a surface of general type.

  Then the continuous map $\pi \colon  Def (S)  \ra \sT (M)_S$ is a local homeomorphism
  between Kuranishi space and Teichm\"uller space if
  
1) $Aut (S)$  is a trivial group, or
  
 2)  $K_S$ is ample and $S$ is rigidified.
\end{theo}

\Proof
We need only to show that $\pi$ is injective. Assume the contrary: then there are two points $t_1, t_2 \in Def(S)$
yielding surfaces $S_1$ and $S_2$ isomorphic through a diffeomorphism $\Psi$ isotopic to the identity.

By the previous remark, the images of $t_1, t_2 $ inside $Def (X) / Aut (X)$ must be the same.

Case 1): there exists then an element $w$ of the Weyl group of $\lambda$ carrying $t_1 $ to $t_2$,
hence the composition of $w$ and $\Psi$ yields an automorphism of $S_1$. Since $Aut(S_1) = Aut (X_1)$
and the locus of canonical models with non trivial automorphisms is closed, we conclude that, taking $Def(S)$ as a suitably
small germ, then this automorphism is the identity. This is however a contradiction,
since $w$ acts non trivially on the cohomology of the exceptional divisor, while $\Psi$ acts trivially.  

Case 2) : In this case there is an automorphism $g$ of $S$ carrying $t_1 $ to $t_2$, and again the composition of $g$ and $\Psi$ yields an automorphism of $S_1$.  We apply the same argument, since $g$ is not isotopic to the identity by our assumption.

\qed

\begin{rem}
With more work one should be able to treat the more general case where we  assume that $Aut (S)$ is non trivial, but $S$ is 
rigidified. In this case one should show that a composition $ g \circ w$ as above is not isotopic to the identity.

The most interesting question is however whether every surface of general type is rigidified.

\end{rem}

\section{Connected components of moduli spaces and arithmetic of moduli spaces for surfaces}

\subsection{Gieseker's moduli space and the analytic moduli spaces}

As we saw, all 5-canonical models of surfaces of general type with invariants $ K^2$, $\chi$ occur
in a big family parametrized by an open set of the Hilbert scheme $ \HHH ^0$ parametrizing 
subschemes with Hilbert polynomial $ P (m)= 
\chi  + \frac{1}{2}( 5 m -1)  5 m  K^2 ,$ namely the open set

$$ \HHH ^0 (\chi, K^2) : = \{ \Sigma | \Sigma {\rm \ is\ reduced \ with \ only \
R.D.P. 's \ as \
singularities \ }\}.
$$

Indeed, it is not necessary to consider the
5-pseudo moduli space of surfaces of general type with
given invariants $ K^2$, $\chi$, which was defined  as the closed subset
$ \HHH _0 \subset \HHH^0$,
$$  \HHH _0 (\chi, K^2) : = \{ \Sigma \in \HHH^0| \omega_{\Sigma}^{
\otimes 5} \cong
   \hol_{\Sigma}(1)   \} .$$
   
   At least, if we are only interested about having a family which contains all surfaces of general type,
   and are not interested about taking the quotient by the projective group.
   
   Observe however that if $\Sigma \in  \HHH ^0 (\chi, K^2)$, then $\Sigma$  is the canonical model
   $X$ of a surface of general type , embedded by a linear system $|D|$, where $D$ is numerically equivalent 
   to $5K_S$, i.e., $D = 5 K_S+ \eta$, where $ \eta$ is numerically equivalent to $0$.
   
   Therefore the connected components $\sN$, respectively the irreducible components $\sZ$ of the Gieseker moduli space
   correspond to the connected , resp.  irreducible, components of $  \HHH _0 (\chi, K^2)$, and in turn to
   the connected , resp.  irreducible, components of $  \HHH ^0 (\chi, K^2) $ which intersect
 $  \HHH _0 (\chi, K^2)$.
 
 We shall however, for the sake of brevity, talk about connected components $\sN$ of the Gieseker moduli space $\frak M^{can}_{a,b}$
 even if these do not really parametrize families of canonical models.
 
We refer to \cite{perugia} for a more ample discussion of the basic ideas which we are going to sketch here.

$\frak M^{can}_{a,b}$ has a finite number of connected components, and 
these parametrize the 
deformation classes
of surfaces of general type. By the classical theorem of Ehresmann 
(\cite{ehre}),
deformation equivalent varieties are diffeomorphic, and moreover, by a diffeomorphism carrying the canonical class 
to the canonical class.

Hence, fixed the two numerical invariants $\chi (S) = a, K^2_S = b$,
which are determined
by the topology of $S$ (indeed, by the  Betti numbers of $S$),
we have a finite number of
differentiable types.

It is clear that the analytic moduli space $\frak M (S)$ that we defined at the onset is then the union of a finite number
of connected components of $\frak M^{can}_{a,b}$. But how many, and how?

A very optimistic guess was: one.

A basic question was really whether a moduli space $\mathfrak M (S)$
would correspond to a unique connected component of the Gieseker moduli space,
and this question was abbreviated as the DEF =  DIFF question.

I.e., the question whether differentiable equivalence and deformation equivalence would coincide for surfaces.

I conjectured (in \cite{katata}) that the answer should be negative, on the basis of some families of simply connected
surfaces of general type constructed in \cite{cat1}: these were then homeomorphic by the results of Freedman (see \cite{free}, and \cite{f-q}),
and it was then relatively easy to show then (\cite{cat3}) that there were many connected components of the moduli space
corresponding to homeomorphic but non diffeomorphic surfaces. It looked like the situation should be similar
even if one would fix the diffeomorphism type.

Friedman and Morgan instead made 
the `speculation' that the answer to the DEF= DIFF question should be positive (1987) (see \cite{f-m1}), motivated by the new examples of homeomorphic 
but not diffeomorphic surfaces discovered by Donaldson (see \cite{don} for a survey on this topic).

The question was finally answered in the negative, and  in every
possible way (\cite{man4},\cite{k-k},\cite{cat4},\cite{c-w},\cite{bcg}.

\begin{theo}
(Manetti '98, Kharlamov -Kulikov 2001, C. 2001, C. - Wajnryb 2004,
Bauer- C. - Grunewald 2005 )

The Friedman- Morgan speculation
does not hold true and the  DEF= DIFF question has a negative answer.
\end{theo}

In my joint work with
Bronek Wajnryb (\cite{c-w}) the question was also shown to have a negative answer even for simply connected surfaces.

I showed later (\cite{cat02}) that each minimal surface of general type $S$ has a natural symplectic structure with
class of the sympletic form equal to $c_1 (K_S)$, and in such a  way that to each connected component $\sN$ of the moduli space one can associate the pair of a  differentiable manifold with a symplectic structure, unique up to symplectomorphism. 

Would this further datum determine a unique connected component, so that  DEF = SIMPL ?

This also turned out to have a negative answer (\cite{cat09}).

\begin{theo}
Manetti surfaces provide  counterexamples to the  DEF = SIMPL question. 
\end{theo}

I refer to  \cite{cime} for a  rather comprehensive treatment of the above questions.

Let me just observe that the Manetti surfaces are not simply connected, so that the DEF=SYMPL question is still open for
the case of simply connected surfaces. Concerning the question of canonical symplectomorphism of algebraic surfaces,
Auroux and Katzarkov (\cite{a-k}) defined asymptotic braid monodromy invariants of a symplectic manifold, extending  old 
ideas of Moishezon (see \cite{moi}). 

Quite recent work, not covered in  \cite{cime}, is my joint work with L\"onne and Wajnryb (\cite{clw}), which  investigates in this direction the braid monodromy invariants (especially the `stable' ones)
for the surfaces introduced in \cite{cat1}.

 \subsection{Arithmetic of moduli spaces}

A basic remark is that all these schemes are defined by equations involving only $\ZZ$ coefficients,
since the defining equation of the Hilbert scheme is a rank condition on a multiplication map
(see for instance \cite{green}), and similarly the condition  $ \omega_{\Sigma}^{
\otimes 5} \cong
   \hol_{\Sigma}(1) $ is also closed (see \cite{abvar}) and defined over $\ZZ$..
   
It follows that  the absolute Galois group  $ Gal (\overline{\QQ}, \QQ)$ acts on the Gieseker
moduli space $\frak M^{can}_{a,b}$.

To explain how it concretely acts, it suffices to recall the
notion of a conjugate variety.

\begin{rem}
1) $\phi \in Aut (\CC)$ acts on $ \CC [z_0, \dots z_n]$,
by sending $P (z) = \sum_{i =0}^n \ a_i z ^i
\mapsto  \phi (P) (z) : = \sum_{i =0}^n \ \phi (a_i) z ^i$.

2) Let $X$ be as above a projective variety
$$X  \subset  \PP^n_\CC,  X : = \{ z | f_i(z) = 0 \ \forall i \}.$$

The action of $\phi$ extends coordinatewise to $ \PP^n_\CC$,
and carries $X$ to another variety, denoted $X^{\phi}$,
and called the {\bf conjugate variety}. Since $f_i(z) = 0 $ implies
$\phi (f_i)(\phi (z) )= 0 $, we see that

$$  X^{\phi}  = \{ w | \phi (f_i)(w) = 0 \  \forall i \}.$$

\end{rem}

If $\phi$ is complex conjugation, then it is clear that the variety
$X^{\phi}$ that we obtain is diffeomorphic  to $X$, but in general,
what happens when $\phi$ is not continuous ?

Observe that, by the theorem of Steiniz,  one has a surjection $ Aut (\CC) \ra Gal(\bar{\QQ} /\QQ)$,
and by specialization the heart of the question concerns the action of $Gal(\bar{\QQ} /\QQ)$
 on varieties $X$ defined over $\bar{\QQ}$.

For curves, since in general the dimensions of spaces of
differential forms of a fixed degree and without poles are the same
for $X^{\phi}$ and $X$, we shall obtain a curve of the same genus,
hence $X^{\phi}$ and $X$ are diffeomorphic.

But for higher dimensional varieties this breaks down,
as discovered by Jean  Pierre Serre in the 60's (\cite{serre}),
who proved the existence of a field automorphism  $\phi \in
Gal(\bar{\QQ} /\QQ)
$, and a variety $X$ defined over $\bar{\QQ}$ such that
$X$ and the Galois conjugate variety $X^{\phi}$  have
  non isomorphic fundamental groups.

In  work   in collaboration with Ingrid Bauer and Fritz
Grunewald (\cite{almeria}, \cite{bcgGalois}) we discovered  wide classes of algebraic surfaces
for which the same phenomenon holds.

A striking result in a similar direction was obtained by Easton and Vakil (\cite{east-vak})

\begin{theo}

The absolute Galois group $Gal(\bar{\QQ} /\QQ)$
acts faithfully
on the  set of irreducible 
components of the (coarse) moduli space of canonical surfaces of general type,
$$ \frak M^{can} : = \cup_{a,b \geq 1} \frak M^{can}_{a,b}. $$

\end{theo}

\subsection{Topology sometimes determines connected components}

  There are cases where the presence of a big fundamental group implies that a connected component 
  of the moduli space is determined by some topological invariants. 
  
  A typical case is the one of surfaces isogenous to a product (\cite{isogenous}),
 where a surface is said to be isogenous to a (higher) product if and only if   it is a quotient
  $$ (C_1 \times C_2)/G, $$
  where $C_1, C_2$ are curves of genera $g_1, g_2 \geq 2$, and $G$ is a finite group acting freely on $ (C_1 \times C_2)$.

  \begin{theo}\label{fabiso}(see \cite{isogenous}). 
  
  a) A
     projective smooth surface is isogenous to a higher product  if and only if
the following two conditions are satisfied:

1) there is an exact sequence
$$
1 \rightarrow \Pi_{g_1} \times \Pi_{g_2} \rightarrow \pi = \pi_1(S)
\rightarrow G \rightarrow 1,
$$
where $G$ is a finite group and where $\Pi_{g_i}$ denotes the fundamental
group of a compact curve of genus $g_i \geq 2$;

2) $e(S) (= c_2(S)) = \frac{4}{|G|} (g_1-1)(g_2-1)$.

\noindent
b) Any surface $X$ with the
same topological Euler number and the same fundamental group as $S$
is diffeomorphic to $S$. The corresponding subset of the moduli space,
$\mathfrak{M}^{top}_S = \mathfrak{M}^{diff}_S$, corresponding to
surfaces orientedly homeomorphic,
resp. orientedly diffeomorphic to $S$, is either
irreducible and connected or it contains
two connected components which are exchanged by complex
conjugation.

In particular, if $S'$ is orientedly diffeomorphic to $S$, then $S'$ is
deformation equivalent to $S$ or to $\bar{S}$.
\end{theo}

Other non trivial examples  are the cases of Keum-Naie surfaces, Burniat surfaces and Kulikov surfaces
(\cite{keumnaie} , \cite{burniat1}, \cite{chancough}): for these classes of surfaces the main result is that any surface
homotopically equivalent to a surface in the given class belongs to a unique irreducible connected component of
the moduli space.

Just to give a flavour of some of the arguments used,  let us consider a simple example which I worked
out together with Ingrid Bauer.

Let $S$ be a minimal surface of general type with $ q(S) \geq 2$.
Then we have the Albanese map
$$\alpha \colon S \ra A: = Alb (S),  $$
and $S$ is said to be of Albanese general type if $\alpha (S) : = Z$ is a surface.
This property is a topological property (see \cite{albanese}), since $\alpha$ induces a homomorphism of cohomology algebras
$$\alpha ^* \colon H^* (A, \ZZ) \ra H^* (S, \ZZ)$$
and $H^* (A, \ZZ) $ is the full exterior algebra  $\Lambda^* (H^1 (A, \ZZ)) \cong \Lambda^* (H^1 (S, \ZZ))$ over $H^1 (S, \ZZ)$.

In particular, in the case where $ q(S)= 2$, the degree $d$ of the Albanese map equals
the index of the image of $\Lambda^4 H^1 (S, \ZZ)$ inside $H^4 (S, \ZZ) = \ZZ [S]$.

The easiest case is the one when $d=2$, because then $K_S = R$, $R$ being the ramification divisor.
Observe that the Albanese morphism factors through the canonical model $X$ of $S$,
and a morphism $ a \colon X \ra A$.

Assume now that $a$ is a finite morphism, so that $2K_X = a^*(a_* (K_X))$.
In particular, if we set $D : = a_* (K_X)$, then $ D^2 = 2 K_X^2 =  2 K_S^2$,
and this number is also a topological invariant.

By the standard formula for double covers we have that $ p_g (S) = h^0 (L) + 1$,
where $D$ is linearly equivalent to $2L$;
hence, if $L$ is a polarization of type $(d_1, d_2)$, then 
$ p_g (S) = d_1 d_2 + 1$, $ D$ is a polarization of type $(2d_1, 2d_2)$,
and $4 d_1 d_2  = 2 L^2 = K_S^2$, hence in particular we have 
$$  K_S^2 = 4 (p_g -1) = 4 \chi(S), $$
since $q(S) = 2$.

I can moreover recover the polarization type $(d_1, d_2)$ (where $d_1$ divides $d_2$) using the fact
that $2 d_1$ is exactly the divisibility index of $D$. 
This is in turn the divisibility of $K_S$, since $K_S$ gives a linear form on $H^2 (A, \ZZ)$
simply by the composition of pushforward and cup product, and this linear form is represented by the class
of $D$.
Finally, the canonical class $K_S$  is a differentiable invariant of $S$ (see \cite{don5} or \cite{mor}).

The final argument is that, by formulae due to  Horikawa (\cite{quintics}), necessarily if $  K_S^2 =  4 \chi(S)$
the branch locus has only negligible singularities (see \cite{keumnaie}), which means that the normal finite cover branched over $D$
has rational double points as singularities.

\begin{theo}{\bf (Bauer-Catanese)}\label{doublecover}
Let $S$ be a minimal surface of general type whose canonical model $X$ is a finite double cover
of an Abelian surface $A$, branched on a divisor $D$ of type $(2d_1, 2 d_2)$.
Then $S$ belongs to an irreducible connected component $\sN$ of the moduli space of dimension $ 4 d_1 d_2 + 2 = 4 \chi(S) + 2$.

Moreover, 

1) any other surface which is diffeomorphic to such a surface $S$ belongs to the component $\sN$.

2) The Kuranishi space $Def(X)$ is always smooth.

\end{theo}

The assumption that $X$ is a finite double cover is a necessary one. 

For instance, Penegini and Polizzi (\cite{pepo2}) construct surfaces
with $p_g (S)= q (S)= 2$ and $ K^2_S=6$ such that for the general surface the canonical divisor is ample (whence $ S=X$),
while the Albanese map, which is generically finite of  degree $2$, contracts an elliptic curve $Z$ with $ Z^2 = -2$ to a point. The authors show then that the
corresponding subset of the moduli space consists of three irreducible connected components.

Other very interesting examples with degree $d=3$ have been considered by Penegini and Polizzi in \cite{pepo}.

 \section{Smoothings and surgeries}
 
 Lack of time prevents me to develop this section.
 
I refer the reader to \cite{cime} for a general discussion, and to the articles \cite{man4} and  \cite{korean} for  interesting applications of the $\QQ$-Gorenstein smoothings
technique (for the study of components of moduli spaces, respectively for the construction of new interesting surfaces).

There is here a relation with the topic of compactifications of moduli spaces. Arguments to show that certain subsets
of the moduli spaces are closed involved taking limits of canonical models and studying certain singularities (see \cite{cat2}, \cite{man3}, see also \cite{riem} for the relevant  results on deformations of singularities); 
in \cite{k-sb} a more general study was pursued of the singularities allowed on the boundary of the moduli space of surfaces.
I refer to the article by Koll\'ar in this Handbook for the general problem of compactifying moduli spaces (also Viehweg devoted a big effort into this enterprise, see \cite{vieh}, \cite{lastopus}, another important reference is \cite{nikos}). 

An explicit study of compactifications of the moduli spaces of surfaces of general type was pursued in \cite{opstall}, \cite{ale-pardini},
\cite{wenfei}, \cite{soenkecomp}. 

 There is here another relation, namely  to the article by Abramovich and others in this Handbook, since the deformation of pairs
 $(Y,D)$ where $Y$ is a smooth complex manifold and $D = \cup_{i=1, \dots, h}D_i$ is a normal crossing divisor, are governed by the
 cohomology groups $$H^i (\Theta_Y (- \log D_1, \dots ,- \log D_h)),$$ for $i=1,2$, and where the sheaf  $\Theta_Y (- \log D_1, \dots ,- \log D_h)$
 is the Serre dual of the sheaf $\Omega^1_Y ( \log D_1, \dots ,\log D_h)(K_{{Y}})$, with its residue sequence 
 $$ 0 \ra \Omega^1_{{Y}} (K_{{Y}}) \ra 
\Omega^1_{\tilde{Y}}(\log D_1, \dots ,\log D_h)
(K_{{Y}}) \ra
\bigoplus_{i=1}^3 \hol_{D_i} (K_{{Y}})\ra 0.
$$

These sheaves are the appropriate ones to calculate the deformations of ramified coverings, see for instance \cite{cat1}),\cite{Pardini}, \cite{bidouble},  \cite{burniat2}, and especially  \cite{burniat3}). 

I was taught about these by David Mumford back in 1977, when he had just been working on the Hirzebruch proportionality
principle in the non compact case (\cite{Hirzprop}).

\bigskip

\noindent
{\bf Acknowledgements.}

I would like to thank Ingrid Bauer, Barbara Fantechi and  Keiji Oguiso for interesting conversations, Wenfei Liu and JinXing Cai
for email correspondence, S\"onke Rollenske, Antonio J. Di Scala and two referees for reading the manuscript with care and pointing out corrections. 

I am grateful to Jun Muk Hwang, Jongnam Lee and Miles Reid for  organizing the winter school `Algebraic surfaces and related topics', which took place at the KIAS in Seoul
in March 2010: some themes treated in those lectures are reproduced in this text.

A substantial part of the paper was written when I was a visiting KIAS scholar in  March-April 2011: I am very grateful for the
excellent atmosphere I found at the KIAS and for  the very nice  hospitality  I received from Jun Muk Hwang
and Jong Hae Keum.



\end{document}